\documentclass[1[leqno,11pt]{amsart}
\usepackage{amssymb,amsmath,amsmath,latexsym,amssymb,amsfonts,amsbsy, amsthm}

\numberwithin{equation}{section}

\allowdisplaybreaks

\let\al=\alpha
\let\a=\alpha

\let\e=\varepsilon

\let\f=\frac

\let\ga=\gamma

\let\Om=\Omega

\let\ol=\overline

\let\na=\nabla

\let\pa=\partial
\let\tre=\triangleq

\def\R{\mathbf R}
\def\T{\mathbf T}
\def\Z{\mathbf Z}
\def\N{\mathbf N}

\def\no{\noindent}

\def\eqdef{\buildrel\hbox{\footnotesize def}\over =}

\newcommand{\beq}{\begin{equation}}
\newcommand{\eeq}{\end{equation}}
\newcommand{\ben}{\begin{eqnarray}}
\newcommand{\een}{\end{eqnarray}}
\newcommand{\beno}{\begin{eqnarray*}}
\newcommand{\eeno}{\end{eqnarray*}}

\newtheorem{theorem}{Theorem}[section]

\newtheorem{lemma}[theorem]{Lemma}
\newtheorem{proposition}[theorem]{Proposition}

\newtheorem{remark}[theorem]{Remark}

\begin{document}

\title[Global well-posedness of the MHD equations]
{Global well-posedness of the MHD equations in a homogeneous magnetic field}

\author{Dongyi Wei}
\address{School of Mathematical Sciences, Peking University, 100871, Beijing, P. R. China}
\email{jnwdyi@163.com}

\author{Zhifei Zhang}
\address{School of Mathematical Sciences, Peking University, 100871, Beijing, P. R. China}
\email{zfzhang@math.pku.edu.cn}

\date{\today}

\begin{abstract}
In this paper, we study the MHD equations with small viscosity and resistivity coefficients, which may be different. This is a typical setting in high temperature plasmas. It was proved that the MHD equations are globally well-posed if the initial velocity is close to 0 and the initial magnetic field is close to a homogeneous magnetic field in the weighted H\"{o}lder space, where the closeness is independent of the dissipation coefficients.
 \end{abstract}

\maketitle

\section{Introduction}
In this paper, we consider the incompressible magneto-hydrodynamics (MHD) equations in $[0,T)\times \Omega, \Omega \subseteq \R^{d}$:
\begin{align}\label{eq:MHD}
\left\{
\begin{array}{l}
\partial_t v-\nu\Delta v+v\cdot\nabla v+\nabla p=b\cdot\nabla b,\\
\partial_t b-\mu\Delta b+v\cdot\nabla b=b\cdot\nabla v,\\
\text{div } v=\text{div } b=0,
\end{array}\right.
\end{align}
where $v$ denotes the velocity field and $b$ denotes the magnetic field, and $\nu\ge 0$ is the viscosity coefficient, $\mu\ge 0$ is the resistivity coefficient.  If $\nu=\mu=0$, (\ref{eq:MHD}) is the so called ideal MHD equations; If $\nu>0$ and $b=0$, (\ref{eq:MHD}) is reduced to the Navier-Stokes equations.  We refer to \cite{Ser} for the mathematical introduction to the MHD equations.

It is well-known that the 2-D MHD equations with full viscosities(i.e., $\nu>0$ and $\mu>0$) have global smooth solution.
In general case, the question of whether smooth solution of the MHD equations develops singularity in finite time is basically open \cite{Ser,CF}. Recently, Cao and Wu \cite{Cao1} studied the global regularity of the 2-D MHD equations with partial dissipation and magnetic diffusion. We refer to \cite{Cao2, Che, Fer, He, Jiu,Lei} and references therein for more relevant results.

In this paper, we are concerned with the global well-posedness of the MHD equations in a homogeneous magnetic field $B_0$.
Recently, there are a lot of works \cite{AZ, LXZ, Ren1, Ren2, Zhang}  devoted to the case without resistivity(i.e, $\nu>0$ and $\mu=0$). Roughly speaking, it was proved that the MHD equations are globally well-posed and the solution decays in time if the initial velocity field is close to 0 and the initial magnetic field is close to $B_0$. These results especially justify the numerical observation \cite{CC}: the energy of the MHD equations is dissipated at a rate independent of the ohmic resistivity.

In high temperature plasmas, both the viscosity coefficient $\nu$ and resistivity coefficient $\mu$ are usually very small \cite{CC}.
Up to now, the heating mechanism of the solar corona is still an unsolved problem in physics \cite{P}. So, it is very interesting to investigate the long-time dynamics of the MHD equations in the case when the dissipation coefficients are very small.

For the simplicity, let us first look at the ideal case. Following \cite{BSS}, we rewrite
the system (\ref{eq:MHD}) in terms of  the Els$\ddot{a}$sser variables
$$Z_+=v+b,\quad Z_-=v-b.$$
Then the ideal MHD equations  (\ref{eq:MHD})  can be written as
\begin{align}\label{eq:MHD-e}
\left\{
\begin{array}{l}
\partial_t Z_++Z_-\cdot\nabla Z_+=-\nabla p,\\
\partial_t Z_-+Z_+\cdot\nabla Z_-=-\nabla p,\\
\text{div} Z_+=\text{div}Z_-=0.
\end{array}\right.
\end{align}
 We introduce the fluctuations
 \beno
 z_+=Z_+-B_0,\quad z_-=Z_-+B_0.
 \eeno
Then the system (\ref{eq:MHD-e}) can be reformulated as
\begin{align}\label{eq:MHD-f}
\left\{
\begin{array}{l}
\partial_t z_++Z_-\cdot\nabla z_+=-\nabla p,\\
\partial_t z_-+Z_+\cdot\nabla z_-=-\nabla p,\\
\text{div} z_+=\text{div}z_-=0.
\end{array}\right.
\end{align}
In the case of  $\Omega=\R^d$,  Bardos-Sulem-Sulem \cite{BSS} proved that  for large time,  the solution $z_\pm$ of (\ref{eq:MHD-f})  tends to linear Alfv\'{e}n waves:
\beno
\pa_tw_\pm\pm B_0\cdot\na w_\pm=0.
\eeno
In two recent works \cite{Cai} and \cite{HXY},  Cai-Lei and He-Xu-Yu studied the global well-posedness of (\ref{eq:MHD})  for any $\nu=\mu\ge 0$ and
$\Om=\R^3$.  The result in \cite{Cai}  also includes the case of $\Omega=\R^2$.

From the physical point of view, it is more natural to consider the MHD equations in a domain with the boundary.  One frequently used domain in physics is a slab bounded by two hyperplanes, i.e., $\Om=\R^{d-1}\times [0,1]$. More importantly,  although both $\nu$ and $\mu$ are very small, they should be different in the real case. However, the proof in \cite{Cai, HXY}  strongly relies on the facts that  $\Omega$ is a whole space and $\nu=\mu$.

The main goal of this paper is to prove the global well-posedness of (\ref{eq:MHD}) in the physical case when $\Om$ is a slab and $\nu\neq \mu$. In this case, we need to impose the suitable boundary conditions
on $z_\pm$.  Let  $z_\pm$ be a function of $(t, x,y), (x,y)\in \Om$. In the case when $\nu=\mu=0$, we impose the nonpenetrating boundary condition
\ben\label{bc:zero}
z^{d}_\pm=0 \quad \text{on}\quad y=0,1.
\een
In the case when $\nu>0$ and $\mu>0$, we impose the Navier-slip  boundary condition
\ben\label{bc:viscosity}
&&z^{d}_\pm=0, \quad \pa_d z_\pm^i=0\quad i=1,\cdots, d-1,\quad \text{on}\quad y=0,1.
\een

To deal with the boundary case,  our idea is to use the symmetric extension and solve the MHD equations in the framework of H\"{o}lder spaces $C^{1,\al}$ for $0<\al<1$. In the ideal case, we give a representation formula of the pressure by using the symmetric extension. Although the extended solution has not the same regularity as the origin one under the nonpenetrating boundary condition, we find that $\na p$ still lies in $C^{1,\al}$ based on the representation formula.
In the viscous case, we can reduce the slab domain to $\Om=\R^{d-1}\times \T$ by using the symmetric extension, because the extended solution still keeps the $C^{1,\al}$  regularity under the Navier-slip boundary condition.

The most challenging task comes from the case of $\nu\neq\mu$.
In the case of $\nu=\mu$, the following formulation plays a crucial role in the proof of \cite{Cai, HXY}:
\begin{align*}
\left\{
\begin{array}{l}
\partial_t z_++Z_-\cdot\nabla z_+=\nu\Delta z_+-\nabla p,\\
\partial_t z_-+Z_+\cdot\nabla z_-=\nu\Delta z_--\nabla p.
\end{array}\right.
\end{align*}
Indeed, the viscosity leads to more technical troubles compared with the ideal case. To handle the case of $\nu\neq\mu$, we need to introduce
many new ideas. First of all, we introduce a key decomposition: let
$\mu_1=\dfrac{\nu+\mu}{2},\ \mu_2=\dfrac{\nu-\mu}{2},$ and we decompose $z_+=z_+^{(1)}+z_+^{(2)},\ z_-=z_-^{(1)}+z_-^{(2)}$ so that
\begin{align*}
\left\{\begin{array}{l}
\partial_t z_+^{(1)}+Z_-\cdot\nabla z_+^{(1)}=\mu_1\triangle z_+^{(1)}-\nabla p_+^{(1)},\\
\partial_t z_-^{(1)}+Z_+\cdot\nabla z_-^{(1)}=\mu_1\triangle z_-^{(1)}-\nabla p_-^{(1)},\\
\partial_t z_+^{(2)}+Z_-\cdot\nabla z_+^{(2)}=\mu_1\triangle z_+^{(2)}+\mu_2\triangle z_--\nabla p_+^{(2)},\\
\partial_t z_-^{(2)}+Z_+\cdot\nabla z_-^{(2)}=\mu_1\triangle z_-^{(2)}+\mu_2\triangle z_+-\nabla p_-^{(2)}.
\end{array}\right.
\end{align*}
The next task is to establish a closed uniform estimate for the solution $z_\pm^{(1)}$ and $z_\pm^{(2)}$ with respect to $\mu_1$ and $t$.
For this end, we need the following key ingredients:
\begin{itemize}

\item The construction of  the weighted H\"{o}lder spaces for the solution.
Due to the appearance of the extra trouble term $\Delta z_\pm$, we have  to work in the spaces with different regularity and weight for the solution $z_\pm^{(1)}$ and  $z_\pm^{(2)}$.  Such incomformity gives rise to the essential difficulties. Especially, the choice of the weight is very crucial.

\item Uniform estimates of the transport equation in the weighted H\"{o}lder spaces, which are used to control the growth of Lagrangian map.

\item Uniform estimates for the parabolic equation with variable coefficients. The proof is based on the uniform estimates of heat operator  in the weighted H\"{o}lder spaces.

\item Boundedness of Riesz transform and its commutator in the weighted H\"{o}lder spaces, which is essentially used to handle the nonlocal pressure. To our knowledge, these results are new, and the proof is highly nontrivial.
\end{itemize}

\smallskip

In this paper, we consider the MHD equations in a homogeneous magnetic field. In the real case(for example, solar corona), it is more natural to consider the MHD equations in an inhomogeneous magnetic field. An important question is to consider the decay of Alfv\'{e}n waves in an inhomogeneous magnetic field $B_0(y)=(b_1(y), b_2(y), 0)$. This is similar to the situation of Landau damping.

\section{The weighted H\"{o}lder spaces and symmetric extension}

\subsection{Weighted H\"{o}lder spaces}

Let $\Omega\subseteq \R^{d}$ be a domain and $\alpha \in (0,1]$. We denote by $C^{k,\alpha}(\Omega),\ (k=0,1)$ the H\"{o}lder space equipped with the norm
\beno
|u|_{0,\al;\Om}\eqdef |u|_{0;\Om}+[u]_{\al;\Om},\quad |u|_{1,\al; \Omega}\eqdef  |u|_{0;\Om}+|\na u|_{0,\al;\Omega},
\eeno
where
\beno
\ |u|_{0;\Om}=\sup_{X\in \Om}|u(X)|,\quad [u]_{\al;\Om}=\sup_{X,Y\in \Om}\frac{|u(X)- u(Y)|}{|X-Y|^{\al}}.
\eeno

Let $h(X)\in C(\R^{d})$ be a positive bounded function. We introduce the following weighted $C^{k,\al}$ norms
\beno
|u|_{0,\al;h, \Om}\eqdef |u|_{0;h,\Om}+[u]_{\al;h, \Om},\quad |u|_{1,\al;h, \Om}\eqdef |u|_{0;h,\Om}+|\nabla u|_{0,\al;h, \Om},
\eeno
where
\beno
|u|_{0;h,\Om}=\big|\frac u h\big|_{0;\Om},\quad [u]_{\al;h, \Om}=\sup_{X,Y\in \Om}\frac{| u(X)- u(Y)|}{(h(X)+h(Y))|X-Y|^{\al}}.
\eeno
We say that $u\in C^{k,\al}_h(\Om)$ if $|u|_{k,\al;h,\Om}<+\infty$. We also introduce
 \beno
|u|_{k,\al;h,\Om,T}\eqdef\sup_{0\leq t \leq T}|u(t)|_{k,\al;h(t),\Om}.
\eeno
When $\Om=\R^{d}$,  we will omit the subscript $\Om$ in the norm of H\"{o}lder spaces.
\smallskip

The following two lemmas can be proved by using the definition of H\"{o}lder norm.
\begin{lemma}\label{lem:product}
Let $h, h_1, h_2$ be the weight functions so that there exists a constant $c_0$ so that
\ben\label{ass:W}
0<c_0h(X)\le h(Y)\quad \text{for any }X,Y\in \R^d,\,\,|X-Y|\le 2.
\een
Then there exists a constant $C$ depending only on $c_0$
so that for $k=0,1$,
\beno
&&|u|_{0,\al; h,\Om}\le C\big(|u|_{0;h,\Om}+|\na u|_{0;h,\Om}\big),\\
&&|uw|_{k,\al; h_1h_2, \Om}\le C|u|_{k,\al;h_1,\Om}|w|_{k,\al;h_2,\Om},\\
&&\left|\int_t^su(r)dr\right|_{k,\al; \int_t^sh(r)dr,\Om}\le \sup\limits_{t\leq r\leq s}|u(r)|_{k,\al; h(r),\Om}.
\eeno
\end{lemma}

\begin{lemma}\label{lem:comp}
Let $\Phi$ be a map from $\Om$ to $\Om$ with $\na \Phi\in C^{0,\al}(\Om)$.
It holds that
\beno
&&|u\circ\Phi|_{0,\a;h\circ\Phi,\Om}\leq |u|_{0,\a;h,\Om}\max\big\{|\nabla\Phi|_{0;\Om}^{\a},1\big\},\\
&&|u\circ\Phi|_{1,\a;h\circ\Phi,\Om}\leq |u|_{1,\a;h,\Om}\max\big\{|\nabla\Phi|_{0;\Om}^{\a},1\big\}\max\big\{|\nabla\Phi|_{0,\a; \Om},1\big\}.
\eeno
Here and in what follows, $|\na \Phi|$ denotes the matrix norm defined by
\ben\label{def:matrix}
|A|\eqdef \sup_{|X|=1}|AX|.
\een
\end{lemma}

To deal with the viscous case, we introduce the following scaled weighted H\"{o}lder space.
Let $\al \in (0,1), R\geq0$ and define
\beno
&&|u|_{0,\a;h,R}\eqdef |u|_{0;h}+R^{\a}[u]_{\a;h},\\
&&|u|_{1,\a;h,R}\eqdef |u|_{0,\al;h}+\max(R, R^{1-\al}\big)|\nabla u|_{0,\a;h,R}.
\eeno
For this kind of weight spaces, we have the analogous of Lemma \ref{lem:product} and Lemma \ref{lem:comp}. For example, if $h(X)$
satisfies
\ben\label{ass:W1}
0<c_0h(X)\le h(Y)\quad \text{for any }X,Y\in \R^d,\,\,|X-Y|\le 2R.
\een Then for $R\ge 1$, we  have
\beno
|u|_{0;h}+R|\nabla u|_{0,\a;h,R}\leq |u|_{1,\a;h,R}\leq |u|_{0,\a;h,R}+R|\nabla u|_{0,\a;h,R}\leq C\big(|u|_{0;h}+R|\nabla u|_{0,\a;h,R}\big).
\eeno
Here $C$ is a constant depending only on $c_0$. In the following, we will fix $\al \in (0,1)$.

\begin{lemma}\label{lem:inhom}
Let $\ga>0$ and $h(X)>0$. Then there exists a constant $C$ independent of $h, \ga,  t$ so that
\begin{align*}
&\Big|\int_0^tu(s)ds\Big|_{1,\a;h,\sqrt{k+\ga t}}\leq C\ga^{-1}\sup\limits_{0<s<t}\Big((\ga s)^{\frac{1}{2}}(\ga (t-s))^{\frac{1}{2}}|u(s)|_{0,\a;h}\\&\quad+\varphi_{\a}(\sqrt{k+\ga s})(\ga (t-s))^{1-\frac{\a}{2}}|\nabla u(s)|_{0;h}+\varphi_{\a}(\sqrt{k+\ga s})(\ga (t-s))^{\frac{3-\a}{2}}[\nabla u(s)]_{1;h}\Big),
\end{align*}
where $\varphi_\al(R)=\max(R,R^{1+\al})$.
\end{lemma}

\begin{proof}
We denote by $C\ga^{-1} A $ the right hand side of the inequality.
Then we have
\beno
\Big|\int_0^tu(s)ds\Big|_{0,\a;h}\leq\int_0^t|u(s)|_{0,\a;h}ds \leq \int_0^t(\ga s)^{-\frac{1}{2}}(\ga (t-s))^{-\frac{1}{2}}dsA\leq C\ga^{-1}A,\eeno\beno\Big|\nabla\int_0^tu(s)ds\Big|_{0;h}\leq\int_0^t|\nabla u(s)|_{0;h}ds \leq \int_0^t\varphi_{\a}(\sqrt{k+\ga s})^{-1}(\ga (t-s))^{-1+\frac{\a}{2}}dsA\\ \leq C\ga^{-1}\min ((k+\ga t)^{-\frac{1}{2}},(k+\ga t)^{-\frac{1-\a}{2}})A.
\eeno
For any $X,Y\in\R^d,$ we have
\beno
&&|\nabla u(s,X)-\nabla u(s,Y)|\leq |X-Y|(h(X)+h(Y))[\nabla u(s)]_{1;h},\\
&&|\nabla u(s,X)-\nabla u(s,Y)|\leq |\nabla u(s,X)|+|\nabla u(s,Y)| \leq (h(X)+h(Y))|\nabla u(s)|_{0;h}.
\eeno
This  gives
\begin{align*}
&|\nabla u(s,X)-\nabla u(s,Y)|\\
&\leq\min((\ga (t-s))^{\frac{1}{2}},|X-Y|)(h(X)+h(Y))\big([\nabla u(s)]_{1;h}+(\ga (t-s))^{-\frac{1}{2}}|\nabla u(s)|_{0;h}\big)\\
&\leq\min((\ga (t-s))^{\frac{1}{2}},|X-Y|)(h(X)+h(Y))\varphi_{\a}(\sqrt{k+\ga s})^{-1}(\ga (t-s))^{-\frac{3-\a}{2}}A,
\end{align*}
therefore,
\begin{align*}
&\Big|\nabla\int_0^tu(s)ds(X)-\nabla\int_0^tu(s)ds(Y)\Big|\leq\int_0^t|\nabla u(s,X)-\nabla u(s,Y)|ds \\
&\leq \int_0^t\min((\ga (t-s))^{\frac{1}{2}},|X-Y|)(h(X)+h(Y))\varphi_{\a}(\sqrt{k+\ga s})^{-1}(\ga (t-s))^{-\frac{3-\a}{2}}Ads\\
&\leq C(h(X)+h(Y)) A\Big(\min((\ga t)^{\frac{1}{2}},|X-Y|)\int_0^{\frac{t}{2}}\varphi_{\a}(\sqrt{k+\ga s})^{-1}ds(\ga t)^{-\frac{3-\a}{2}}\\
&\qquad+\int_{\frac{t}{2}}^t\min((\ga (t-s))^{\frac{1}{2}},|X-Y|)(\ga (t-s))^{-\frac{3-\a}{2}}ds\varphi_{\a}(\sqrt{k+\ga t})^{-1}\Big)\\
&\leq C(h(X)+h(Y)) A\Big((\ga t)^{\frac{1-\a}{2}}|X-Y|^{\a}t\varphi_{\a}(\sqrt{k+\ga t})^{-1}(\ga t)^{-\frac{3-\a}{2}}\\&\qquad+\ga^{-1}|X-Y|^{\a}\varphi_{\a}(\sqrt{k+\ga t})^{-1}\Big)\\
&\leq C\ga^{-1}(h(X)+h(Y)) A|X-Y|^{\a}\varphi_{\a}(\sqrt{k+\ga t})^{-1}.
\end{align*}
Summing up, we deduce our result.\end{proof}

\begin{lemma}\label{lem:comp1}
Let $\Phi$ be a map from $\R^d$ to $\R^d$ with $\na \Phi\in C^{0,\al}(\R^d)$.
It holds that
\beno
&&|u\circ\Phi|_{0,\a;h\circ\Phi,R}\leq |u|_{0,\a;h,R}\max\big(|\nabla\Phi|_{0}^{\a},1\big),\\
&&|u\circ\Phi|_{1,\a;h\circ\Phi,R}\leq |u|_{1,\a;h,R}\max\big(|\nabla\Phi|_{0}^{\a},1\big)\max\big(|\nabla\Phi|_{0,\a;1,R},1\big).
\eeno
\end{lemma}

\subsection{Symmetric extension}

Let $\Omega=\R^{d-1}\times [0,1]$ be a strip and $X=(x,y), x\in \R^{d-1}, y\in [0,1]$ be a point in $\Om$.

Let $T_e$ be an even extension from ${C({\Om})}$ to $C(\R^{d})$ defined by
\beno
T_ef(x,2n+y)=T_ef(x,2n-y)=f(x,y)
\eeno
for $x\in \R^{d-1}, y\in [0,1], n\in \Z.$
Let $T_o$ be an odd extension from $C_0(\Om)=\big\{u\in C({\Om}), u=0 \text{ on } \pa\Om\big\}$ to $C(\R^{d})$ defined by
\beno
T_of(x,2n-y)=-f(x,y),\quad T_of(x,2n+y)=f(x,y)
\eeno
for $x\in \R^{d-1}, y\in [0,1], n\in \Z.$

\begin{lemma}\label{lem:extension}
It holds that
\beno
&&|T_ef|_{0,\al}=|f|_{0,\al,\Om},\\
&&|f|_{0,\al;\Om}\le |T_of|_{0,\al}\le 2|f|_{0,\al;\Om}.
\eeno
The same result holds for the weighted H\"{o}lder norm $|\cdot|_{0,\al;h}$ if the weight function $h(X)$ depends only on $x$.
\end{lemma}

\begin{proof}
First of all, it is obvious that
\beno
|f|_{0,\al;\Om}\leq |T_ef|_{0,\al},\quad
|f|_{0,\al;\Om}\le |T_of|_{0,\al},
\eeno
and the same is true for the weighted H\"{o}lder norm $|\cdot|_{0,\al;h}$.
We denote
\beno
&&\rho_0(y)=\inf\limits_{n\in\Z}|y-2n|\in [0,1]\quad \text{for}\ y\in\R,\\
&&\rho(X)=(x,\rho_0(y))\in \Omega\quad\text{for}\ X=(x,y)\in\R^{d},
\eeno
and let
\beno
 \Omega_{+}=\bigcup\limits_{n\in\Z}\R^{d-1}\times [2n,2n+1],\quad  \Omega_{-}=\bigcup\limits_{n\in\Z}\R^{d-1}\times [2n-1,2n].
\eeno
Then it is easy to see that
\beno
&&T_ef=f\circ\rho,\\
&&T_of=f\circ\rho\quad \text{in}\ \Omega_{+},
\quad T_of=-f\circ\rho\quad \text{in} \ \Omega_{-},\\
&&|\rho_0(y)-\rho_0(y')|\leq |y-y'|,\quad |\rho(X)-\rho(Y)|\leq |X-Y|,
\eeno
from which, it follows that
\beno
&&|T_ef|_{0,\al}\leq|f|_{0,\al;\Om} ,\quad |T_ef|_{0,\al;h}\leq|f|_{0,\al;h,\Om},\\&& |T_of|_{0}\leq |f|_{0;\Om},\quad |T_of|_{0;h}\leq |f|_{0;h,\Om}.
\eeno

Given $X=(x,y),\ Y=(x',y')\in \R^{d}$ with $y\leq y'$, if $X,Y\in \Omega_{+}$ or $X,Y\in \Omega_{-}$, then
\begin{align*}
|T_of(X)-T_of(Y)|&=|f\circ\rho(X)-f\circ\rho(Y)|\\
&\leq |f|_{0,\a;h,\Omega}(h\circ\rho(X)+h\circ\rho(Y))|\rho(X)-\rho(Y)|^{\a}\\
&\leq |f|_{0,\a;h,\Omega}(h(X)+h(Y))|X-Y|^{\a}.
\end{align*}
Here we used $h\circ\rho(X)=h(X)$.
Otherwise, there exists $y_1,y_2\in\Z$ so that $y_1-1\leq y\leq y_1\leq y_2\leq y'\leq y_2+1$. Let $X'=(x,y_1),\ Y'=(x',y_2)$. Then for $f\in C_0(\Omega)$, we have
\begin{align*}
|T_of(X)|=|f\circ\rho(X)|=&|f\circ\rho(X)-f\circ\rho(X')|\\
\leq& |f|_{0,\a;h,\Omega}(h\circ\rho(X)+h\circ\rho(X'))|\rho(X)-\rho(X')|^{\a}\\\leq& 2|f|_{0,\a;h,\Omega}h(X)|X-X'|^{\a}.
\end{align*}
Similarly, we have
\beno
|T_of(Y)|\leq 2|f|_{0,\a;h,\Omega}h(Y)|Y-Y'|^{\a}.
\eeno
Then, using $|X-X'|+|Y-Y'|\leq |X-Y|$, we get
\beno
|T_of(X)-T_of(Y)|\leq 2|f|_{0,\a;h,\Omega}(h(X)+h(Y))|X-Y|^{\a}.
\eeno
This shows  $[T_of]_{\a;h}\leq 2[f]_{\a;h,\Omega}$. Similarly, $[T_of]_{\a}\leq 2[f]_{\a;\Omega}$.
\end{proof}

\section{Global well-posedness for the ideal MHD equations}

This section is devoted to the proof of global well-posedness of
the ideal MHD equations in $\R^{d-1}\times [0,1]$ with the boundary condition \eqref{bc:zero}. Recall that in terms of the Elsasser variables $z_\pm=Z_\pm\pm B_0$, the ideal MHD equations take as follows
\begin{align}\label{eq:MHD-ideal}
\left\{
\begin{array}{l}
\partial_t z_++Z_-\cdot\nabla z_+=-\nabla p,\\
\partial_t z_-+Z_+\cdot\nabla z_-=-\nabla p,\\
\text{div} z_+=\text{div}z_-=0,\\
z_\pm^d(t,x,y)=0\quad \text{on}\quad y=0,1.
\end{array}\right.
\end{align}

Without loss of generality, we take the background magnetic field $B_0=(1,0,\cdots,0)$.

\subsection{Main result}

Let $f(x, y)=f_0(x_1)$, where $f_0\in C^1(\R)$ is chosen so that $|f_0'|<f_0<1$ and for some $C_1>0$,
\ben\label{ass:f}
\begin{split}
 &\delta(T)\triangleq\sup_{Y\in \R^d}\int_{-T}^{T}f(Y+ 2B_0t)dt {\leq C_1}\quad\text{for any}\quad T>0,\\
 &\int_{\R^d}\frac{f(Y)}{1+|X-Y|^{d+1}}dY\leq C_1f(X)\quad \text{for any}\quad X\in\R^d,\\
 &f(X)\leq 2f(Y)\quad\text{for any}\quad |X-Y|\leq 2.
 \end{split}
 \een
In fact, $f_0(r)=(C_0+r^2)^{-{\frac{\delta+1}{2}}}$ satisfies the above conditions for some $C_0>1$ and  ${0<\delta<1}$.

Now we introduce the weight function $f_\pm(t,X)$ given by
\beno
f_\pm(t,X)\tre f(X\pm B_0t),
\eeno
which satisfies \eqref{ass:W} with a uniform constant $c_0$ independent of $t$. Let
\beno
M_\pm(t)\triangleq \sup_{|s|\le t}|z_\pm(s)|_{1,\al;f_\pm(s),\Om}.
\eeno

Main result of this section is stated as follows.

\begin{theorem}\label{thm:ideal}
Let $\al\in (0,1)$.
There exists $\epsilon>0$ so that if $M_\pm(0)\le \epsilon$,
then there exists a global in time unique solution $(z_+,z_-)\in L^\infty\big(0,+\infty;C^{1,\al}(\Om)\big)$ with the pressure $p$ determined by (\ref{eq:pressure}) to  the ideal MHD equations (\ref{eq:MHD-ideal}), which  satisfies
\beno
M_\pm(t)\le C\epsilon\quad \text{for any}\quad t\in [0,+\infty).
\eeno
\end{theorem}

\begin{remark}
Thanks to $M_\pm(0)\sim \big|z_\pm(0)\langle x_1\rangle^{1+\delta}\big|_{1,\al;\Om}$ if $f_0(r)=(C_0+r^2)^{-\frac{\delta+1}{2}}$, this means that the initial data decays at infinity only in one direction. This is a key point for the
global well-posedness in the strip domain,
especially in $\R\times [0,1]$.
\end{remark}

We conclude this subsection by introducing some properties
of weight functions.  Let
\beno
g(t,X)\tre\int_{\R^d}\frac{f(Y+B_0t)f(Y-B_0t)}{1+|X-Y|^{d+1}}dY.
\eeno
We have the following important facts.

\begin{lemma}\label{lem:weight}
There exists a constant $C>0$ so that for any $X\in \R^d, t\in \R$,
\beno
&&f(X+B_0t)f(X-B_0t)\leq Cg(t,X),\\
&&g(t,X)\leq C\big(1+|X-Y|\big)^{d+1}g(t,Y),\\
&&\int_{-T}^{T}g(t,X\pm B_0t)dt\leq C\delta(T)f(X).
\eeno
\end{lemma}
\begin{proof}
Thanks to $f(Y)\geq f(X)/2$ for $|X-Y|<2$,  we get
\begin{align*}
g(t,X)&\geq\int_{B(X,2)}\frac{f(Y+B_0t)f(Y-B_0t)}{1+|X-Y|^{d+1}}dY\\
&\geq\f 14\int_{B(X,2)}\frac{f(X+B_0t)f(X-B_0t)}{1+|X-Y|^{d+1}}dY\\
&\geq C^{-1}{f(X+B_0t)f(X-B_0t)},
\end{align*}
which gives the first inequality.

Using the inequality $\dfrac{1}{1+|X-Z|^{d+1}}\leq C\dfrac{1+|X-Y|^{d+1}}{1+|Y-Z|^{d+1}}$, we infer
\begin{align*}
g(t,X)=&\int_{\R^d}\frac{f(Z+B_0t)f(Z-B_0t)}{1+|X-Z|^{d+1}}dZ\\
\leq& C\int_{\R^d}\frac{f(Z+B_0t)f(Z-B_0t)}{1+|Y-Z|^{d+1}}(1+|X-Y|^{d+1}) dY\\
=& C\big(1+|X-Y|^{d+1}\big)g(t,Y),
\end{align*}
which gives the second inequality.

Make a change of variable
\beno
g(t,X+B_0t)=\int_{\R^d}\frac{f(Y+B_0t)f(Y-B_0t)}{1+|X+B_0t-Y|^{d+1}}dY =\int_{\R^d}\frac{f(Y+2B_0t)f(Y)}{1+|X-Y|^{d+1}}dY,
\eeno
which along with (\ref{ass:f})  gives
\begin{align*}
\int_{-T}^{T}g(t,X+B_0t)=&\int_{\R^d}\frac{\int_{-T}^{T}f(Y+2B_0t)f(Y)dt}{1+|X-Y|^{d+1}}dY\\
\leq& C\int_{\R^d}\frac{\delta(T)f(Y)}{1+|X-Y|^{d+1}} dY \leq C \delta(T)f(X).
\end{align*}
Similarly, we have
\beno
\int_{-T}^{T}g(t,X-B_0t) \leq C \delta(T)f(X).
\eeno
This completes the proof of the lemma.
\end{proof}

\subsection{Weighted $C^{1,\al}$ estimate for the transport equation}

Let $Z\in C^1\big([0,T]\times {\Om}\big)$ be a vector field with
$Z^d=0$ on $\pa\Om$. We introduce the characteristic associated with $Z$:
\ben
\frac { d} {dt}\Phi(s,t,X)=Z(t,\Phi(s,t,X)),\quad  \Phi(s,s,X)=X.
\een
Then $\Phi(s,t,X)\in C^1\big([0,T]\times[0,T]\times {\Om}\big)$
is a diffeomorphism from $\Om$ to $\Om$ and $\pa\Om$ to $\pa\Om$
having the property
\beno
\Phi(r,t)\circ \Phi(s,r)=\Phi(s,t),\quad \Phi(s,s)=Id.
\eeno

\begin{lemma}\label{lem:char}
If $Z(t,X)$ satisfies the extra condition
\begin{align}\label{eq:Z-cond}
\ |\nabla Z|_{0,\a;h,{\Omega}, T}\int_{t_0}^Th\big(t, \Phi(T,t,X)\big) dt\leq A_0\quad\text{for any}\quad  X\in {\Omega},
\end{align}
then it holds that for $0\underline{\leq t_0}\leq t\leq s<T$
\beno
&& |\nabla \Phi(s,t)-Id|_{0;{\Omega}}\leq e^{A_0}-1,\\
&&|\nabla \Phi(s,t)|_{0;{\Omega}}\leq e^{A_0},\\
&&[\nabla \Phi(s,t)]_{\a;{\Omega}}\leq 2A_0e^{(2+\a)A_0}.
\eeno
\end{lemma}

\begin{proof}
Thanks to the definition of $\Phi(s,t)$, we have
\beno
&& \partial_t\nabla\Phi(s,t)=\nabla\Phi(s,t)((\nabla Z(t))\circ\Phi(s,t)),\\
&&\Phi(s,s)=Id,\quad \nabla\Phi(s,s)=Id,\\
&&\ |\nabla\Phi(s,t)|\leq |\nabla\Phi(s,t)-Id|+1.
\eeno
Here $|\na \Phi(s,t)|$ is the matrix norm defined by (\ref{def:matrix}).
Therefore,
\begin{align*}
|\nabla\Phi(s,t)-Id|\leq& \int_t^s|\partial_r\nabla\Phi(s,r)|dr\\
\leq& \int_t^s|\nabla\Phi(s,r)||(\nabla Z(r))\circ\Phi(s,r)|dr\\
\leq& \int_t^s|(\nabla Z(r))\circ\Phi(s,r)|dr+\int_t^s|\nabla\Phi(s,r)-Id||(\nabla Z(r))\circ\Phi(s,r)|dr,
\end{align*}
which implies
\beno
|\nabla\Phi(s,t)-Id|\leq \exp\Big(\int_t^s|(\nabla Z(r))\circ\Phi(s,r)|dr\Big)-1.
\eeno
Thanks to
\beno
|(\nabla Z(r))\circ\Phi(s,r)|\leq |\nabla Z|_{0,\a;h,{\Omega}, T}h(r)\circ\Phi(s,r),
\eeno
we get by (\ref{eq:Z-cond}) that
\begin{align*}
  \int_t^s|(\nabla Z(r))\circ\Phi(s,r)(X)|dr\leq& |\nabla Z|_{0,\a;h,{\Omega}, T}\int_t^sh(r)\circ\Phi(s,r)(X)dr\\
=&|\nabla Z|_{0,\a;h,{\Omega}, T}\int_t^sh(r,\Phi(T,r,\Phi(s,T)(X)))dr\\
\leq& A_0.
 \end{align*}
Thus, we conclude that
 \beno
 &&|\nabla\Phi(s,t)-Id|_{0;{\Omega}}\leq e^{A_0}-1,\\
 &&|\nabla\Phi(s,t)|_{0;{\Omega}}\leq e^{A_0},\\
 &&|\Phi(s,t,X)-\Phi(s,t,Y)|\leq |\nabla\Phi(s,t)|_{0;{\Omega}}|X-Y|\leq e^{A_0}|X-Y|.
 \eeno

Notice that
\begin{align*}
|\nabla\Phi(s,t,X)-&\nabla\Phi(s,t,Y)|\leq \int_t^s|\nabla\Phi(s,r,X)-\nabla\Phi(s,r,Y)||(\nabla Z(r))\circ\Phi(s,r,X)|dr\\
&+\int_t^s|\nabla\Phi(s,r,Y)||(\nabla Z(r))\circ\Phi(s,r,X)-(\nabla Z(r))\circ\Phi(s,r,Y)|dr,
\end{align*}
from which and Gronwall's inequality, we infer
\beno
&&|\nabla\Phi(s,t,X)-\nabla\Phi(s,t,Y)|\\ &&
\leq \int_t^s|\nabla\Phi(s,r,Y)||(\nabla Z(r))\circ\Phi(s,r,X)-(\nabla Z(r))\circ\Phi(s,r,Y)|dr \exp\Big(\int_t^s|(\nabla Z(r))\circ\Phi(s,r,X)|dr\Big)\\
&&\leq \int_t^s|\nabla\Phi(s,r,Y)||\nabla Z|_{0,\a;h,{\Omega}, T}(h(r,\Phi(s,r,X))+h(r,\Phi(s,r,Y)))|\Phi(s,r,X)-\Phi(s,r,Y)|^{\a}dre^{A_0}\\
&&\leq \int_t^se^{A_0}|\nabla Z|_{0,\a;h,{\Omega}, T}(h(r,\Phi(s,r,X))+h(r,\Phi(s,r,Y)))e^{\a A_0}|X-Y|^{\a}dre^{A_0}\\
&&=e^{(2+\a)A_0}|X-Y|^{\a}|\nabla Z|_{0,\a;h,{\Omega}, T}\int_t^s(h(r,\Phi(s,r,X))+h(r,\Phi(s,r,Y)))dr\\
&&\leq 2A_0e^{(2+\a)A_0}|X-Y|^{\a},
\eeno
which shows the last inequality of the lemma.
\end{proof}

Next we consider the transport equation
\ben\label{eq:trans-D}
\partial_t u+Z\cdot\nabla u=F,\quad u(0,X)=u_0(X).
\een
Using the characteristic, the solution $u(t,X)$ is given by
\ben\label{eq:trans-I}
u(t,X)=u_0(\Phi(t,0,X))+\int_0^tF\big(s,\Phi(t,s,X)\big)ds.
\een

\begin{lemma}\label{lem:tran-H}
If $Z$ satisfies (\ref{eq:Z-cond}), then we have
\beno
&&|u(t)|_{0,\a;{\Omega}}\leq e^{\a A_0}\Big(|u_0|_{0,\a;{\Omega}}+\int_0^t|F(s)|_{0,\a;{\Omega}}ds\Big),\\
&&|\text{div}\,u(t)|_{0;{\Omega}}\leq |\text{div}\,u_0|_{0;{\Omega}}+\int_0^t|(tr(\nabla Z\nabla u)- \text{div} F)(s)|_{0;{\Omega}}ds.
\eeno
\end{lemma}

\begin{proof}
Using (\ref{eq:trans-I}), Lemma \ref{lem:comp} and Lemma \ref{lem:char}, we get
\begin{align*}
|u(t)|_{0,\a;{\Omega}}\leq& |u_0\circ\Phi(t,0)|_{0,\a;{\Omega}}+\int_0^t|F(s)\circ\Phi(t,s)|_{0,\a;{\Omega}}ds\\
\leq & |u_0|_{0,\a;{\Omega}}\max\big\{|\nabla\Phi(t,0)|_{0;\Om}^{\a},1\big\}+\int_0^t|F(s)|_{0,\a;{\Omega}}
\max\big\{|\nabla\Phi(t,s)|_{0;\Om}^{\a},1\big\}ds\\
\leq& e^{\a A_0}\Big(|u_0|_{0,\a;{\Omega}}+\int_0^t|F(s)|_{0,\a;{\Omega}}ds\Big).
\end{align*}

Taking divergence to (\ref{eq:trans-D}),  we obtain
 \beno
\partial_t\text{div} u+Z\cdot\nabla\text{div} u+tr(\nabla Z\nabla u)=\text{div}F,\quad u(0,X)=u_0(X).
\eeno
So, we have
\beno
\text{div}u(t)=\text{div}u_0\circ\Phi(t,0)+\int_0^t(\text{div}F-tr(\nabla Z\nabla u))(s)\circ\Phi(t,s)ds,
\eeno
then the second inequality follows easily.
\end{proof}

\begin{proposition}\label{prop:tran-HW}
If $|Z+B_0|_{1,\a;f_-,{\Omega}, T}\delta(T)<1,$ then we have
\beno
|u|_{1,\a;f_+,{\Omega},T}\leq C\big(|u_0|_{1,\a;f,{\Omega}}+\delta(T)|F|_{1,\a;g,{\Omega},T}\big).
\eeno
If $| Z-B_0|_{1,\a;f_+,{\Omega}, T}\delta(T)<1,$ then we have
\beno
|u|_{1,\a;f_-,{\Omega},T}\leq C\big(|u_0|_{1,\a;f,{\Omega}}+\delta(T)|F|_{1,\a;g,{\Omega},T}\big).
\eeno
Here $C$ is a constant independent of $T$.
\end{proposition}

\begin{proof}
We only prove the first inequality, the proof of the second one is similar.
Let us claim
\ben\label{claim:phi}
|\Phi(s,t,X)+B_0(t-s)-X|<2\quad \text{for}\quad  0\leq t\leq s\leq T.
\een
Otherwise, there exists $t\in [0,s]$ such that $|\Phi(s,t,X)+B_0(t-s)-X|=2$ and $|\Phi(s,r,X)+B_0(r-s)-X|\leq2$ for $r\in[t,s]$. Thus,
\begin{align*}
 |\Phi(s,t,X)+B_0(t-s)-X|\leq& \int_t^s|\partial_r\Phi(s,r,X)+B_0|dr\\
 =&\int_t^s|Z(r,\Phi(s,r,X))+B_0|dr\\
 \leq&\int_t^s|Z+B_0|_{1,\a;f_-,{\Omega}, T}f_-(r,\Phi(s,r,X))dr\\
 =&|Z+B_0|_{1,\a;f_-,{\Omega}, T}\int_t^sf(\Phi(s,r,X)-B_0r)dr,
 \end{align*}
while, by \eqref{ass:f},
\beno
\int_t^sf(\Phi(s,r,X)-B_0r)dr\leq 2\int_t^sf(X-B_0(r-s)-B_0r)dr\leq 2\delta(T).\eeno
This shows
\beno
|\Phi(s,t,X)+B_0(t-s)-X|\leq 2|Z+B_0|_{1,\a;f_-,{\Omega}, T}\delta(T)<2,
\eeno
which is a contradiction, hence \eqref{claim:phi} is true.

Now we verify (\ref{eq:Z-cond}) for $h=f_-,\ A_0=2$. Indeed, by \eqref{ass:f} and \eqref{claim:phi},
\begin{align*}
\int_0^Tf_-(t,\Phi(T,t,X))dt=&\int_0^Tf(\Phi(T,t,X)-B_0t)dt\\
\leq& 2\int_0^Tf(X-B_0(t-T)-B_0t)dt\leq 2\delta(T),
\end{align*}
which implies (\ref{eq:Z-cond}). Then we infer from Lemma \ref{lem:char}
that
\ben\label{eq:phi-bound}
|\nabla\Phi(t,s)|_{0,\a; \Om}\leq C.
\een

It follows from Lemma \ref{lem:weight} and \eqref{claim:phi} that
\beno
 \int_0^t g(r,\Phi(t,r,X)) dr\leq C\int_0^t g(r,X-B_0(r-t)) dr\leq C\delta(T)f(X+B_0t),
 \eeno
which implies
 \begin{align*}
|u(t)|_{1,\a;f_+(t),\Omega}\leq& |u_0\circ\Phi(t,0)|_{1,\a;f_+(t),\Omega}+C\delta(T)\sup_{0\leq s\leq t}|F(s)\circ\Phi(t,s)|_{0,\a;g(s)\circ\Phi(t,s),{\Omega}}.
\end{align*}
Using the fact $f(\Phi(t,0,X))\leq 2f(X-B_0(0-t))=2f_+(t,X)$, we get
\beno
|u_0\circ\Phi(t,0)|_{1,\a;f_+(t),\Omega}\le 2|u_0\circ\Phi(t,0)|_{1,\a;f\circ\Phi(t,0),\Omega}.
\eeno
Then by Lemma \ref{lem:comp}  and  \eqref{eq:phi-bound}, we obtain\begin{align*}
|u(t)|_{1,\a;f_+(t),\Omega}\leq & C\Big(|u_0\circ\Phi(t,0)|_{1,\a;f,\Omega}+\delta(T)\sup_{0\leq s\leq t}|F(s)|_{1,\a;g(s),{\Omega}}\Big)\\&\quad\times\max\big\{|\nabla\Phi(t,s)|_{0;\Om}^{\a},1\big\}\max\big\{|\nabla\Phi(t,s)|_{0,\a; \Om},1\big\}\\ \leq& C|u_0|_{1,\a;f,\Omega}+C\delta(T)\sup_{0\leq s\leq t}|F(s)|_{1,\a;g(s),{\Omega}}.
\end{align*}
This shows the first inequality of the lemma.
\end{proof}

\subsection{Representation formula of the pressure}

In this subsection, we give a representation formula of the pressure
by using the symmetric extension.

Let $(v,b,p)$ be a smooth solution of (\ref{eq:MHD}) in $[0,T]\times \Om$
with the boundary condition (\ref{bc:zero}).  We make the following symmetric extension for the solution:
\beno
\overline{v}=Tv\eqdef \big(T_ev^1,\cdots,T_ev^{d-1},T_ov^{d}\big),\quad
\overline{b}=Tb,\quad \overline{p}=T_ep.
\eeno
Then $\big(\overline{v}, \overline{b},\overline{p}\big)$ satisfies (\ref{eq:MHD}) in $[0,T]\times \R^{d}$ in the weak sense.
Although the solution after symmetric extension has not the same smoothness as the origin one, we have the following important observation.

\begin{lemma}\label{lem:extension-p}
Let $h$ be a weight satisfying \eqref{ass:W}.
Let $u=(u^1,\cdots, u^d), w=(w^1,\cdots, w^d)\in C^{1,\al}_h(\Om)$ be two vector fields with $u^{d}=w^{d}=0$ on $\pa\Om$. Let
$\overline{u}=Tu,\,\overline{w}=Tw.$
Then it holds that for $i,j=1,\cdots, d$,
\begin{align*}
&|\partial_i\overline{u}^j\partial_j\overline{w}^i|_{0,\al;h}+|\partial_i\overline{u}^i\partial_j\overline{w}^j|_{0,\al;h} \leq C|\nabla{u}|_{0,\al; h,\Omega}|\nabla{w}|_{0,\al;h,\Omega},\\
&|\overline{u}^j\partial_j\overline{w}^i|_{0,\al;h}+|\overline{u}^i\partial_j\overline{w}^j|_{0,\al;h}\leq C|{u}|_{0,\al;h,\Omega}|\nabla{w}|_{0,\al;h,\Omega}.
\end{align*}
\end{lemma}

\begin{proof}
It is easy to verify that
\beno
&&\partial_i\overline{u}^j\partial_j\overline{w}^i=T_e(\partial_i{u}^j\partial_j{w}^i),\quad \partial_i\overline{u}^i\partial_j\overline{w}^j=T_e(\partial_i{u}^i\partial_j{w}^j),\\
&&\overline{u}^j\partial_j\overline{w}^i=T_e({u}^j\partial_j{w}^i),\quad \overline{u}^i\partial_j\overline{w}^j=T_e({u}^i\partial_j{w}^j)\quad \text{for}\quad i=1,\cdots,d-1\\
&&\overline{u}^j\partial_j\overline{w}^{d}=T_o({u}^j\partial_j{w}^{d}),
\quad \overline{u}^d\partial_j\overline{w}^j=T_0({u}^d\partial_j{w}^j).
\eeno
Then the lemma follows easily from Lemma \ref{lem:extension}.
\end{proof}

Taking the divergence to the first equation of (\ref{eq:MHD}), we get
\beno
-\Delta \overline{p}=\pa_i\big(\overline{v}^j\pa_j\overline{v}^i-\overline{b}^j\pa_j\overline{b}^i\big).
\eeno
Formally, we have
\beno
\na \overline{p}(t,X)=\na\int_{\R^{d}}N(X-Y)\pa_i\big(\overline{v}^j\pa_j\overline{v}^i-\overline{b}^j\pa_j\overline{b}^i\big)(t,Y)dY,
\eeno
where $N(X)$ is the Newton potential.
In terms of the Els$\ddot{a}$sser variables $\overline{z}_\pm(t,X)$, we have
\begin{align*}
\nabla\overline{p}(t,X)=\nabla\int_{\R^{d}}N(X-Y)\partial_i(\overline{z}^j_+\partial_j\overline{z}^i_-)(t,Y)dY.
\end{align*}
However, this integral does not make sense for $\partial_i(\overline{z}^j_+\partial_j\overline{z}^i_-)\in C^{0,\al}$. To overcome this trouble, we introduce a smooth cut-off function $\theta(r)$ so that
\ben\label{def:theta}
\theta(r)=\left\{\begin{array}{l}1\quad \text{for}\ |r|\leq 1,\\0\quad \text{for}\ |r|\geq 2.\end{array}\right.
\een
Integration by parts, we can split $\nabla \ol{p}(t,X)$ as
\begin{align}
-\nabla \ol{p}(t,X)=&\int_{\R^d}\nabla N(X-Y)(\partial_i\overline{z}^j_+\partial_j\overline{z}^i_-)(t,Y)dY\nonumber\\
&+\int_{\R^{d}}\partial_i\partial_j\Big(\nabla N(X-Y)(1-\theta(|X-Y|))\Big)
(\overline{z}^j_+\overline{z}^i_-)(t,Y)dY.\label{eq:pressure}
\end{align}
It is easy to check that this representation makes sense for $\overline{z}_\pm\in W^{1,\infty}(\R^{d})$.

We denote
\begin{align}\label{eq:B(u,v)}
{T_1u}\tre&\int_{\R^d}\nabla N(X-Y)\theta(|X-Y|)u(Y)dY\nonumber\\
{T_{ij}w}\tre&\int_{\R^d}\partial_i\partial_j\Big(\nabla N(X-Y)(1-\theta(|X-Y|))\Big)w(Y)dY.
\end{align}
Let $u,w\in C^{1,\al}(\Om)$ be two vector fields with $u^d=w^d=0$ on $\pa\Om$.
Let $\ol{u}=Tu, \ol{w}=Tw$ be the symmetric extension. We denote
\ben\label{def:I}
I(u,w)\tre T_1(\pa_i\ol{u}^j\pa_j\ol{w}^i-\pa_j\ol{u}^j\pa_i\ol{w}^i)+T_{ij}(\ol{u}^i\ol{w}^j).
\een
Here and in what follows, the repeated index denotes the summation.
Thanks to
\ben
\pa_i\ol{u}^j\pa_j\ol{w}^i-\pa_j\ol{u}^j\pa_i\ol{w}^i=\pa_i\big(\ol{u}^j\pa_j\ol{w}^i-\ol{u}^i\pa_j\ol{w}^j\big),\label{eq:relation-3}
\een
we infer from Lemma \ref{lem:integral} and Lemma \ref{lem:extension-p} that
\begin{align}
|I(u,w)|_{0,\a;\Om}\le C|u|_{0,\al;\Om}|w|_{1,\al;\Om}.\label{eq:I-Cal}\end{align}

Using Lemma \ref{lem:B-div} and \eqref{eq:relation-3}, let us calculate
\begin{align*}
&\text{div}I(u,w)+(\partial_iu^j\partial_jw^i-\partial_iu^i\partial_jw^j)\\
&=\int_{\R^d}\nabla N(X-Y)\cdot\nabla\theta(|X-Y|)
\big(\partial_i\ol{u}^j\partial_j\ol{w}^i-\partial_i\ol{u}^i\partial_j\ol{w}^j\big)(Y)dY\\
&\quad-\int_{\R^d}\partial_i\partial_j\Big(\nabla N(X-Y)\cdot\nabla\theta(|X-Y|)\Big)(\ol{u}^j\ol{w}^i)(Y)dY\\
&=\int_{\R^d}\partial_i\Big(\nabla N(X-Y)\cdot\nabla\theta(|X-Y|)\Big)
\big(-\ol{u}^j\partial_j\ol{w}^i+\ol{u}^i\partial_j\ol{w}^j+\partial_j(\ol{u}^j\ol{w}^i)\big)(Y)dY\\
&=\int_{\R^d}\partial_i\Big(\nabla N(X-Y)\cdot\nabla\theta(|X-Y|)\Big)
\big(\ol{u}^i\text{div} \ol{w}+\ol{w}^i\text{div}\ol{u}\big)(Y)dY,
\end{align*}
which  implies
\ben\label{eq:I-div}
\big|\text{div}I(u,w)-(\partial_iu^j\partial_jw^i-\partial_iu^i\partial_jw^j)\big|_{0;\Om}
\leq C\big(|u|_{0;\Om}|\text{div} w|_{0;\Om}+|w|_{0,\Om}|\text{div}u|_{0;\Om}\big).
\een

In the case of $\R^d$,  the pressure $p(t,X)$ can also be expressed as
\ben\label{pressure-Rd}
-\na p(t,X)=I(z_+,z_-),
\een
where
\begin{align}
I(u,w)\tre &\int_{\R^d}\nabla N(X-Y)\theta(|X-Y|)(\partial_iu^j\partial_jv^i)(Y)dY\nonumber\\
&+\int_{\R^{d}}\partial_i\partial_j\Big(\nabla N(X-Y)(1-\theta(|X-Y|))\Big)
(u^jv^i)(Y)dY.\label{I(u,w)-Rd}
\end{align}
Notice that the representation formula \eqref{pressure-Rd}  is independent of the choice of $\theta$
in $I(u,w)$.

\subsection{Proof of Theorem \ref{thm:ideal}}
Since we can not find a well-posedness theory for the ideal MHD equations in the weighted H\"{o}lder spaces, we will present a complete proof of
Theorem \ref{thm:ideal}. In fact, we find that
the proof of the existence part is very nontrivial.\smallskip

Using the representation of the pressure (\ref{eq:pressure}), we rewrite the system (\ref{eq:MHD-ideal}) as
\ben\label{eq:MHD-ideal-2}
\left\{\begin{array}{l}
\partial_t z_++Z_-\cdot\nabla z_+=-I({z_+},{z_-}),\\
\partial_t z_-+Z_+\cdot\nabla z_-=-I({z_+},{z_-}),\\
z_+(0,X)=z_{+0}(X),\quad z_-(0,X)=z_{-0}(X).
\end{array}\right.
\een

Let $T>0$ be determined later and $A_1=|z_{+0}|_{1,\a;f,\Omega}+|z_{-0}|_{1,\a;f,\Omega}$.
When $A_1$ is sufficiently small, $T$ can be taken $+\infty$.
The system (\ref{eq:MHD-ideal-2}) is solved by the following iteration scheme:
$$z_+^{(0)}=z_-^{(0)}=0,\quad Z_+^{(n)} =z_+^{(n)}+B_0,\quad Z_-^{(n)}=z_-^{(n)}-B_0.$$
Let us inductively assume that $z_\pm^{(n)}$ satisfies
\beno
|z_+^{(n)}|_{1,\a;f_+,{\Omega}, T}\leq 2C_1A_1,\quad |z_-^{(n)}|_{1,\a;f_-,{\Omega}, T}\leq 2C_1A_1.
\eeno
Take $T>0$ so that $4C_1A_1\delta(T)<1$. Then we have
\ben\label{eq:zn-est}
| z_+^{(n)}|_{1,\a;f_+,{\Omega}, T}\delta(T)<\f12,\quad | z_-^{(n)}|_{1,\a;f_-,{\Omega}, T}\delta(T)<\f12.
\een
Now, the solution $z_+^{(n+1)},\ z_-^{(n+1)}$ are determined by
\begin{align*}
\left\{\begin{array}{l}
\partial_t z_+^{(n+1)}+Z_-^{(n)}\cdot\nabla z_+^{(n+1)}=-I({z_+^{(n)}},{z_-^{(n)}}),\\
\partial_t z_-^{(n+1)}+Z_+^{(n)}\cdot\nabla z_-^{(n+1)}=-I({z_+^{(n)}},{z_-^{(n)}}),\\
z_+^{(n+1)}(0,X)=z_{+0}(X),\quad z_-^{(n+1)}(0,X)=z_{-0}(X).
\end{array}\right.
\end{align*}
It follows from Proposition \ref{prop:tran-HW}  that
\beno
&&|z_+^{(n+1)}|_{1,\a;f_+,{\Omega}, T}\leq C_1\big(|z_{+0}|_{1,\a;f,\Omega}+\delta(T)| z_+^{(n)}|_{1,\a;f_+,{\Omega}, T}| z_-^{(n)}|_{1,\a;f_-,{\Omega},T}\big),\\
&&| z_-^{(n+1)}|_{1,\a;f_-,{\Omega}, T}\leq C_1\big(|z_{-0}|_{1,\a;f,\Omega}+\delta(T)| z_+^{(n)}|_{1,\a;f_+,{\Omega}, T}| z_-^{(n)}|_{1,\a;f_-,{\Omega}, T}\big).
\eeno
Here we used
\begin{align*}
|I(u,w)|_{1,\al;g,\Om}\le& C|\pa_i\ol{u}^j\pa_j\ol{w}^i-\pa_j\ol{u}^j\pa_i\ol{w}^i|_{0,\al;h}+C|\ol{u}\ol{w}|_{0;h}\\
\le& C|u|_{1,\al;h,\Om}|w|_{1,\al;h,\Om},
\end{align*}
which follows from Lemma \ref{lem:integral} with $h(t,X)=f_+f_-(t,X)$
and Lemma \ref{lem:extension-p}.

Due to (\ref{eq:zn-est}), we obtain
\beno
|z_+^{(n+1)}|_{1,\a;f_+,{\Omega}, T}\leq 2C_1A_1,\quad |z_-^{(n+1)}|_{1,\a;f_-,{\Omega}, T}\leq 2C_1A_1.
\eeno
In particular, we show that for any $n$,
\beno
| z_+^{(n)}|_{1,\a;f_+, {\Omega},T}\leq C,\quad | z_-^{(n)}|_{1,\a;f_-,{\Omega}, T}\leq C.
\eeno

Next, we show that $\big\{z_\pm^{(n)}\big\}_{n\ge 0}$ is a Cauchy sequence in $C^{0,\al}(\Om)$.
Indeed, we have
\begin{align*}
&\partial_t (z_+^{(n+1)}-z_+^{(n)})+Z_-^{(n)}\cdot\nabla (z_+^{(n+1)}-z_+^{(n)})+(z_-^{(n)}-z_-^{(n-1)})\cdot\nabla z_+^{(n)}\\
&\qquad\qquad+I({z_+^{(n)}}-{z_+^{(n-1)}},{z_-^{(n)}})
+I({z_+^{(n-1)}},{z_-^{(n)}}-{z_-^{(n-1)}})=0,\\
&\partial_t (z_-^{(n+1)}-z_-^{(n)})+Z_+^{(n)}\cdot\nabla (z_-^{(n+1)}-z_-^{(n)})+(z_+^{(n)}-z_+^{(n-1)})\cdot\nabla z_-^{(n)}\\
&\qquad\qquad+I({z_+^{(n)}}-{z_+^{(n-1)}},{z_-^{(n)}})
+I({z_+^{(n-1)}},{z_-^{(n)}}-{z_-^{(n-1)}})=0,\\
&(z_+^{(n+1)}-z_+^{(n)})(0,X)=0,\ (z_-^{(n+1)}-z_-^{(n)})(0,X)=0.
\end{align*}
Then it follows from Lemma \ref{lem:tran-H} and (\ref{eq:I-Cal}) that
\begin{align*}
|(z_+^{(n+1)}-z_+^{(n)})(t)|_{0,\a;\Omega}\leq& C\int_0^t|(z_-^{(n)}-z_-^{(n-1)})(s)|_{0,\a;\Omega}|\nabla z_+^{(n)}(s)|_{0,\a;\Omega}ds\\
&+C\int_0^t|(z_+^{(n)}-z_+^{(n-1)})(s)|_{0,\a;\Omega}|z_-^{(n)}(s)|_{1,\a;\Omega}ds\\
&+C\int_0^t|(z_-^{(n)}-z_-^{(n-1)})(s)|_{0,\a;\Omega}|z_+^{(n-1)}(s)|_{1,\a;\Omega}ds\\
\leq& C_2\int_0^t\big(|(z_+^{(n)}-z_+^{(n-1)})(s)|_{0,\a;\Omega}+|(z_-^{(n)}-z_-^{(n-1)})(s)|_{0,\a;\Omega}\big)ds.
\end{align*}
Similarly, we have
\begin{align*}
|(z_-^{(n+1)}-z_-^{(n)})(t)|_{0,\a;\Omega}\leq C_2\int_0^t\big(|(z_+^{(n)}-z_+^{(n-1)})(s)|_{0,\a;\Omega}+|(z_-^{(n)}-z_-^{(n-1)})(s)|_{0,\a;\Omega}\big)ds.
\end{align*}
This implies that
\beno
|(z_+^{(n+1)}-z_+^{(n)})(t)|_{0,\a;\Omega}+|(z_-^{(n+1)}-z_-^{(n)})(t)|_{0,\a;\Omega}\leq C(2C_2t)^n/n!.
\eeno
Therefore,$z_+^{(n)}, z_-^{(n)}$ converge to some $z_+, z_-$ uniformly in $[0,t]\times\Omega$ for any $0<t<T$. As $z_+^{(n)}, z_-^{(n)}$ are uniformly bounded in $C^{1,\a}$, we have $z_+, z_-\in C^{1,\a}$. Then $\nabla z_+^{(n)}, \nabla z_-^{(n)}$ converge to $\nabla z_+,\ \nabla z_-$ uniformly in $[0,t]\times\Omega$ for any $0<t<T$. Using the equations of $ z_+^{(n+1)}, z_-^{(n+1)}$, $\partial_t z_+^{(n)}, \partial_t z_-^{(n)}$ also converge uniformly in $[0,t]\times\Omega$ for any $0<t<T$. Thus, $z_+,\ z_-\in C^1([0,t]\times\overline{\Omega})$ satisfies (\ref{eq:MHD-ideal-2}) and $z_+^d= z_-^d=0$ on $\partial\Omega$.

Finally, it remains to prove that if $\text{div}z_{+0}=\text{div}z_{-0}=0,$ then $\text{div}z_+=\text{div}z_-=0.$
It follows from Lemma \ref{lem:tran-H} and (\ref{eq:I-div}) that
\begin{align*}
|\text{div}z_+(t)|_{0;\Omega}\leq& \int_0^t|(\partial_iz_+^j\partial_jz_-^i-\text{div}I({z_+},{z_-}))(s)|_{0;\Omega}ds\\
\leq& C\int_0^t\big(|\text{div}z_+(s)|_{0;\Omega}|\text{div}z_-(s)|_{0;\Omega}+|z_+(s)|_{0;\Omega}|\text{div}z_-(s)|_{0;\Omega}
\\&\qquad+|\text{div}z_+(s)|_{0;\Omega}|z_-(s)|_{0;\Omega}\big)ds\\
\leq & C\int_0^t\big(|\text{div}z_+(s)|_{0;\Omega}+|\text{div}z_-(s)|_{0;\Omega}\big)ds,
\end{align*}
similarly,
\begin{align*}
|\text{div}z_-(t)|_{0;\Omega}\leq C\int_0^t(|\text{div}z_+(s)|_{0;\Omega}+|\text{div}z_-(s)|_{0;\Omega})ds.
\end{align*}
This implies that $\text{div}z_+=\text{div}z_-=0.$

Let us remark that  $I(z_+,z_-)$ can be expressed as $\nabla p$. Indeed, we can find $\theta_1,\ \theta_2\in C^\infty(0,+\infty)$  such that $\theta_1'(r)=-\theta(r)N(r),\ \theta_2'(r)=(\theta(r)-1)N(r)$. Let $\theta_{ij}(X)=\partial_i\partial_j\theta_2(|X|)$ and
\begin{align*}
I_*(u,w)(x)=&\int_{\R^d}\theta_1(|X-Y|)(\partial_iu^j\partial_jw^i-\partial_ju^j\partial_iw^i)(Y)dY
\nonumber\\&+\int_{\R^d}(\theta_{ij}(X-Y)-\theta_{i,j}(-Y))
(u^jw^i)(Y)dY.
\end{align*}
Then we have $\nabla I_*(u,v)=I(u,v)$. Therefore, we can take $p=I_*(\overline{z}_+,\overline{z}_-)$, which satisfies $|p|\leq C\ln(2+|x|)$.
This completes the proof of Theorem \ref{thm:ideal}.

\section{Global well-posedness for the viscous MHD  equations}

In this section, we study the global well-posedness for the viscous MHD equations in the slab domain $\Om=\R^{d-1}\times [0,1]$ with the Navier-slip boundary condition. Because we can reduce the slab domain $\Om=\R^{d-1}\times [0,1]$ to  $\R^{d-1}\times \T$ by using the symmetric extension, we will consider more general domain $\Om=\R^k\times \T^{d-k}$ for $2\le k\le d$. The case of $k=1$ is more difficult and will be dealt in the future work.

In fact, $\Om=\R^k\times \T^{d-k}$ is a special case of $\R^d$ periodic in $ d-k$ directions $ e_1,\cdots, e_{d-k}$.
We will assume that $ e_1,\cdots, e_{d-k},B_0$ are linearly independent.

\subsection{New formulation}
Let $\mu_1=\dfrac{\nu+\mu}{2}$ and $\mu_2=\dfrac{\nu-\mu}{2}$.
In terms of the Els$\ddot{a}$sser variables $Z_\pm=v\pm b$, the MHD equations (\ref{eq:MHD}) read
\begin{align}\label{eq:MHD-v1}
\left\{\begin{array}{l}
\partial_t z_++Z_-\cdot\nabla z_+=\mu_1 \triangle z_++\mu_2\triangle z_--\nabla p,\\
\partial_t z_-+Z_+\cdot\nabla z_-=\mu_1 \triangle z_-+\mu_2\triangle z_+-\nabla p,\\
\text{div} z_+=\text{div}z_-=0,
\end{array}\right.
\end{align}
where $z_\pm=Z_+\pm B_0.$ In the case of $\nu=\mu$(thus, $\mu_2=0$), the formulation (\ref{eq:MHD-v1})
plays a crucial role in the proof of \cite{Cai,HXY}. To deal with the case of $\nu\neq\mu$,
we need to introduce the following key decomposition
\beno
z_+=z_+^{(1)}+z_+^{(2)},\quad z_-=z_-^{(1)}+z_-^{(2)},
\eeno
where $z_\pm^{(1)}$ and $z_\pm^{(2)}$ are determined by
\begin{align}\label{eq:MHD-z1}
\left\{\begin{array}{l}
\partial_t z_+^{(1)}+Z_-\cdot\nabla z_+^{(1)}=\mu_1\triangle z_+^{(1)}-\nabla p_+^{(1)},\\
\partial_t z_-^{(1)}+Z_+\cdot\nabla z_-^{(1)}=\mu_1\triangle z_-^{(1)}-\nabla p_-^{(1)},\\
\text{div} z_+^{(1)}=\text{div}z_-^{(1)}=0,\\
 z_+^{(1)}(0)=z_+(0),\ z_-^{(1)}(0)=z_-(0),
 \end{array}\right.
\end{align}
and
\begin{align}\label{eq:MHD-z2}
\left\{\begin{array}{l}
\partial_t z_+^{(2)}+Z_-\cdot\nabla z_+^{(2)}=\mu_1\triangle z_+^{(2)}+\mu_2\triangle z_--\nabla p_+^{(2)},\\
\partial_t z_-^{(2)}+Z_+\cdot\nabla z_-^{(2)}=\mu_1\triangle z_-^{(2)}+\mu_2\triangle z_+-\nabla p_-^{(2)},\\
\text{div} z_+^{(2)}=\text{div} z_-^{(2)}=0,\\
z_+^{(2)}(0)=z_-^{(2)}(0)=0.
\end{array}\right.
\end{align}

To estimate $z_\pm^{(1)}$, we rewrite (\ref{eq:MHD-z1}) as
\begin{align}\label{eq:MHD-z1-z}
\left\{\begin{array}{l}
\partial_t z_+^{(1)}+Z_-^{(1)}\cdot\nabla z_+^{(1)}=\mu_1\triangle z_+^{(1)}-z_-^{(2)}\cdot\nabla z_+^{(1)}-I(z_-^{(2)},z_+^{(1)})-I(z_-^{(1)},z_+^{(1)}),\\
\partial_t z_-^{(1)}+Z_+^{(1)}\cdot\nabla z_-^{(1)}=\mu_1\triangle z_-^{(1)}-z_+^{(2)}\cdot\nabla z_-^{(1)}-I(z_+^{(2)},z_-^{(1)})-I(z_+^{(1)},z_-^{(1)}),
 \end{array}\right.
\end{align}
where $I(u,w)$ is defined by (\ref{I(u,w)-Rd}).
And we also need to use the equation of $J_\pm^{(1)}=\text{curl} z_\pm^{(1)}$, which is given by
\begin{align}\label{eq:MHD-z1-j}
\left\{\begin{array}{l}\partial_t J_+^{(1)}+Z_-^{(1)}\cdot\nabla J_+^{(1)}+\nabla z_-^{(1)}\wedge\nabla z_+^{(1)}+\text{curl}(z_-^{(2)}\cdot\nabla z_{+}^{(1)})=\mu \triangle J_+^{(1)},\\ \partial_t J_-^{(1)}+Z_+^{(1)}\cdot\nabla J_-^{(1)}+\nabla z_+^{(1)}\wedge\nabla z_-^{(1)}+\text{curl}(z_+^{(2)}\cdot\nabla z_{-}^{(1)})=\mu \triangle J_-^{(1)}.
\end{array}\right.
\end{align}
Here $A\wedge B=(AB)-(AB)^T$ is understood as matrix multiplication.

To estimate $z_\pm^{(2)}$, we need to introduce another formulation in terms of the stream function
$\psi_\pm^{(2)}=\Delta^{-1}\text{curl }z_\pm^{(2)}$, which satisfies
\begin{align}
\left\{\begin{array}{l}
\partial_t \psi_{+}^{(2)}+\triangle^{-1}\text{curl}(Z_-\cdot\nabla z_+^{(2)})=\mu_1\triangle \psi_{+}^{(2)}+\mu_2J_-,\nonumber\\
\partial_t \psi_{-}^{(2)}+\triangle^{-1}\text{curl}(Z_+\cdot\nabla z_-^{(2)})=\mu_1\triangle \psi_{-}^{(2)}+\mu_2J_+,
\end{array}\right.
\end{align}
where
\ben\label{eq:j-dec}
J_\pm=\text{curl }z_\pm=J_\pm^{(1)}+\text{curl }z_\pm^{(2)}.
\een

We introduce
\beno
&&\textrm{II}_1(u,w)\tre\triangle^{-1}\text{curl}\ \text{div}(u\otimes w),\\
&&\textrm{II}_2(u,w)\tre \triangle^{-1}\text{curl}(u\cdot\nabla w)-u\cdot\nabla \triangle^{-1}\text{curl}w.
\eeno
So, we get
\begin{align*}
\triangle^{-1}\text{curl}(Z_-\cdot\nabla z_+^{(2)})=&Z_-^{(1)}\cdot\nabla \psi_+^{(2)}+\textrm{II}_1(z_-^{(2)},z_+^{(2)})+\textrm{II}_2(z_-^{(1)},z_+^{(2)}).
\end{align*}
Then we deduce that
\begin{align}\label{eq:MHD-psi}
\left\{\begin{array}{l}
\partial_t\psi_{+}^{(2)}+Z_-^{(1)}\cdot\nabla \psi_{+}^{(2)}+\textrm{II}_2(z_-^{(1)},z_+^{(2)})+\textrm{II}_1(z_-^{(2)}, z_+^{(2)})
=\mu_1\triangle \psi_{+}^{(2)}+\mu_2J_-,\\
\partial_t \psi_{-}^{(2)}+Z_+^{(1)}\cdot\nabla \psi_{-}^{(2)}+\textrm{II}_2(z_+^{(1)},z_-^{(2)})+\textrm{II}_1(z_+^{(2)}, z_-^{(2)})=\mu_1\triangle \psi_{-}^{(2)}+\mu_2J_+.
\end{array}\right.
\end{align}

A direct calculation shows
\beno
&&-\big(\triangle^{-1}\text{curl}(u\cdot\nabla w)\big)^{jk}=\triangle^{-1}(\partial_k\partial_i(u^iw^j)-\partial_j\partial_i(u^iw^k))=-R_kR_i(u^iw^j)+R_jR_i(u^iw^k),\\
&&-\big(u\cdot\nabla(\triangle^{-1}\text{curl }w)\big)^{jk}=u^i\partial_i\triangle^{-1}(\partial_kw^j-\partial_jw^k)=u^i(-R_iR_kw^j+R_iR_jw^k),
\eeno
where $R_i$ is the Riesz transform defined by $R_i=\pa_i(-\Delta)^{-\f12}$. This gives
\ben\label{eq:Pi-2}
\textrm{II}_2(u,w)^{jk}=[u^i,R_iR_j]w^k-[u^i,R_iR_k]w^j.
\een

\subsection{Weighted $C^{1,\al}$ estimates for the parabolic equation}

We consider the parabolic equation with variable coefficients
\ben\label{eq:para}
\partial_tu-\gamma\partial_i\big(a_{ij}\partial_j u\big)+F_1+F_2+\partial_iG^i=0,
\een
where $\gamma>0$ and the coefficients $a_{ij}(t,X)$ satisfies
\ben\label{ass:coef}
\sup_{t\in [0,T]}\big(|a_{ij}(t)-\delta_{ij}|_0+(1+\gamma t)^{\a/2}[a_{ij}(t)]_{\a}\big)\leq \varepsilon_0,
\een
for some $\al\in (0,1), \e_0>0$ and $T>0$.

Let $f(t,X), h(t,X)$ be two weight functions satisfying \eqref{ass:W} with a uniform constant $c_0$ independent of $t$ and
\ben\label{ass:h}
\int_0^tH(2\ga(t-s))h(s,X)ds\leq c_0^{-1}f(t,X),\quad H(2\gamma(t-s))f(s,X)\le c_0^{-1}f(t,X)
\een
for all $0\le {s< t}\leq T,\,X\in \R^d$,
where
\beno
H(t)\phi(X)=\frac{1}{(4\pi t)^{d/2}}\int_{\R^d}e^{-\frac{|X-Y|^2}{4t}}\phi(Y)dY.
\eeno

Let $\delta>0$. We introduce
\begin{align*}
\Lambda_1(T,&F_1,F_2,G,f,h)\tre \sup_{0< t\leq T}\Big(|F_1(t)|_{1,\a;h(t),(1+\ga t)^{1/2}}\\
&+\ga^{-1}\big((\ga t)^{\frac{1}{2}}+(\ga t)^{1+\frac{\delta}{2}}\big)|F_2(t)|_{0,\al;f(t)}
+\ga^{-1}(1+\ga t)^{\frac{1}{2}}|G(t)|_{0,\a;f(t),(1+\ga t)^{1/2}}\Big),
\end{align*}
and
\begin{align*}
\Lambda_0(T,&F_1,F_2,G,f,h)\tre\sup_{0< t\leq T}\Big(|F_1(t)|_{1,\a;h(t),(\ga t)^{1/2}}\\
&+\ga^{-1}\big((\ga t)^{1-\frac{\a}{2}}+(\ga t)^{1+\frac{\delta}{2}}\big)|F_2(t)|_{0,\a;f(t)}
+\ga^{-1}((\ga t)^{\frac{1}{2}}+(\ga t)^{\frac{1-\a}{2}})|G(t)|_{0,\a;f(t),(\ga t)^{1/2}}\Big).
\end{align*}

\begin{proposition}\label{prop:parabolic}
There exist $\varepsilon_0>0$ and $C>0$ independent of $\ga$ and $T$ such that
\begin{align*}
&\sup_{0< t\leq T}|u(t)|_{1,\a;f(t),(1+\ga t)^{1/2}}\leq C\big(|u(0)|_{1,\a;f(0){,1}}+\Lambda_1(T,F_1,F_2,G,f,h)\big),\\
&\sup_{0< t\leq T}|u(t)|_{1,\a;f(t),(\ga t)^{1/2}}\leq C\big(|u(0)|_{0,\a;f(0)}+\Lambda_0(T,F_1,F_2,G,f,h)\big).
\end{align*}

\end{proposition}

\begin{proof}
Let us first consider the case of $a_{ij}=\delta_{ij}$. Then we get
\beno
u(t)=H(\ga t)u(0)+\int_0^t\big(H(\ga (t-s))(F_1(s)+F_2(s))+\partial_iH(\ga (t-s))G^i(s)\big)ds.
\eeno
Using $H(2\gamma t)f(0)\le c^{-1}_0f(0,X)$, we get by  Lemma \ref{lem:heat2} that
\beno
&&|H(\ga t)u(0)|_{1,\a;f(t),(1+\ga t)^{1/2}}\leq C|H(\ga t)u(0)|_{1,\a;H(2\ga t)f(0),(1+\ga t)^{1/2}}\leq C|u(0)|_{1,\a;f(0),1},\\
&&|H(\ga t)u(0)|_{1,\a;f(t),(\ga t)^{1/2}}\leq C|H(\ga t)u(0)|_{1,\a;H(2\ga t)f(0),(\ga t)^{1/2}}\leq C|u(0)|_{0,\a;f(0)}.
\eeno
By  \eqref{ass:h} and Lemma \ref{lem:heat2}, we have
\beno
&&\left|\int_0^tH(\ga (t-s))F_1(s)ds\right|_{1,\a;f(t),(k+\ga t)^{1/2}}\\
&&\leq C\sup_{0<s<t}|H(\ga (t-s))F_1(s)|_{1,\a;H(2\ga(t-s))h(s),(k+\ga t)^{1/2}}\\
&&=C\sup_{0<s<t}|H(\ga (t-s))F_1(s)|_{1,\a;H(2\ga(t-s))h(s),(k+\ga s+\ga (t-s))^{1/2}}\\
&&\leq C\sup_{0<s<t}|F_1(s)|_{1,\a;h(s),(k+\ga s)^{1/2}},
\eeno
and by Lemma \ref{lem:heat2},
\beno
&&\left|\int_0^tH(\ga (t-s))F_2(s)ds\right|_{1,\a;f(t),(k+\ga t)^{1/2}}\\
&&\leq C\int_0^t|H(\ga (t-s))F_2(s)|_{1,\a;H(2\ga(t-s))f(s),(k+\ga t)^{1/2}}ds\\ &&\leq C\int_0^t\varphi_{\a}(\sqrt{k+\ga t})/\varphi_{\a}(\sqrt{\ga(t-s)})|F_2(s)|_{0,\al;f(s)}ds
\eeno
for $k=0,1$. Recall that $\varphi_\al(R)=\max(R,R^{1+\al})$, for $k=0,1$,
\begin{align*}
&\int_0^t\varphi_{\a}(\sqrt{\ga(t-s)})^{-1}\min((\ga s)^{-1+\frac{\a+k(1-\a)}{2}},(\ga s)^{-1-\frac{\delta}{2}})ds\\
&\leq \int_0^t({\ga(t-s)})^{-\frac{1}{2}}(\ga s)^{-1+\frac{\a+k(1-\a)}{2}}ds\\
&\leq C\ga^{-1}(\ga t)^{-\frac{(1-k)(1-\a)}{2}},
\end{align*}
and
\begin{align*}
&\int_0^t\varphi_{\a}(\sqrt{\ga(t-s)})^{-1}\min\big((\ga s)^{-1+\frac{\a+k(1-\a)}{2}},(\ga s)^{-1-\frac{\delta}{2}}\big)ds\\
&\leq \int_0^{t}({\ga(t-s)})^{-\frac{1+\a}{2}}\min\big((\ga s)^{-1+\frac{\a+k(1-\a)}{2}},(\ga s)^{-1-\frac{\delta}{2}}\big)ds\\
&\leq C\int_0^{\frac{t}{2}}({\ga t})^{-\frac{1+\a}{2}}\min\big((\ga s)^{-1+\frac{\a+k(1-\a)}{2}},(\ga s)^{-1-\frac{\delta}{2}}\big)ds+ \int_{\frac{t}{2}}^{t}({\ga(t-s)})^{-\frac{1+\a}{2}}(\ga t)^{-1}ds
\\&\leq C\ga^{-1}(\ga t)^{-\frac{1+\a}{2}}.
\end{align*}
Thus, we have
\beno
&&\int_0^t\varphi_{\a}(\sqrt{k+\ga t})/\varphi_{\a}(\sqrt{\ga(t-s)})\min\big((\ga s)^{-1+\frac{\a+k(1-\a)}{2}},(\ga s)^{-1-\frac{\delta}{2}}\big)ds\\
&&\le C\ga^{-1}\max\big((k+\ga t)^\f12 ,(k+\ga t)^{\f {1+\al} 2}\big)
\min\big((\ga t)^{-\frac{(1-k)(1-\a)}{2}},(\ga t)^{-\frac{1+\a}{2}}\big)\\
&&\le C\ga^{-1}.
\eeno
Therefore, we deduce that for $k=0,1$ and $j=1, 2$,
\beno
 \left|\int_0^tH(\ga (t-s))F_j(s)ds\right|_{1,\a;f(t),(k+\ga t)^{1/2}}\leq C\Lambda_k(T,F_1,F_2,G,f,h).
\eeno

It follows from Lemma \ref{lem:inhom} and Lemma \ref{lem:heat1} that
for $k=0,1$,
\beno
&&\left|\int_0^t\partial_iH(\ga (t-s))G^i(s))ds\right|_{1,\a;f(t),(k+\ga t)^{1/2}}\\
&&\leq C\ga^{-1}\sup\limits_{0<s<t}\Big((\ga s)^{\frac{1}{2}}(\ga (t-s))^{\frac{1}{2}}|\partial_iH(\ga (t-s))G^i(s)|_{0,\a;f(t)}\\
&&\qquad+\varphi_{\a}(\sqrt{k+\ga s})(\ga (t-s))^{1-\frac{\a}{2}}|\nabla \partial_iH(\ga (t-s))G^i(s)|_{0;f(t)}\\&&\qquad+\varphi_{\a}(\sqrt{k+\ga s})(\ga (t-s))^{\frac{3-\a}{2}}[\nabla \partial_iH(\ga (t-s))G^i(s)]_{1;f(t)}\Big)\\
 &&\leq C\ga^{-1}\sup\limits_{0<s<t}((\ga s)^{\frac{1}{2}}|G(s)|_{0,\a;f(s)}+\varphi_{\a}(\sqrt{k+\ga s})[G(s)]_{\a;f(s)}) \\ &&\leq C\ga^{-1}\sup\limits_{0<s<t}((k+\ga s)^{\frac{1}{2}}+(k+\ga s)^{\frac{1-\a}{2}})|G(s)|_{0,\a;f(s),(k+\ga s)^{\frac{1}{2}}}\\
 &&\leq C\Lambda_k(T,F_1,F_2,G,f,h).
 \eeno

 Summing up, we conclude the proof for the case $a_{ij}=\delta_{ij}$.

 To deal with the general case,  we rewrite (\ref{eq:para})  as
\beno
\partial_tu-\gamma\triangle u+F_1+F_2+\partial_i\widehat{G}^i=0,
 \eeno
 where $\widehat{G}^i={G}^i-\ga(a_{ij}-\delta_{ij})\partial_j u$. Thus, we have
  \beno
  \sup_{0< t\leq T}|u(t)|_{1,\a;f(t),(k+\ga t)^{1/2}}\leq C\big(|u(0)|_{1,\a;f(0)\underline{,k}}+\Lambda_k(T,F_1,F_2,\widehat{G},f,h)\big),
  \eeno
for $k=0,1$, where
 \begin{align*}
  \Lambda_k(T,F_1,F_2,\widehat{G},f,h)&\leq \Lambda_k(T,F_1,F_2,{G},f,h)\\&+\sup_{0\leq t\leq T}\sup_{i}((k+\ga t)^{\frac{1}{2}}+(\ga t)^{\frac{1-\a}{2}})|(a_{ij}-\delta_{ij})\partial_j u(t)|_{0,\a;f(t),(k+\ga t)^{\frac{1}{2}}},
 \end{align*}
 and by \eqref{ass:coef},
 \begin{align*}
\big|(a_{ij}-\delta_{ij})\partial_j u(t)\big|_{0,\a;f(t),(k+\ga t)^{\frac{1}{2}}}\leq& C|a_{ij}(t)-\delta_{ij}|_{0,\a;1,(k+\ga t)^{\frac{1}{2}}}|\partial_j u(t)|_{0,\a;f(t),(k+\ga t)^{\frac{1}{2}}}\\
\leq& C\varepsilon_0|\nabla u(t)|_{0,\a;f(t),(k+\ga t)^{\frac{1}{2}}}\\\leq& C\varepsilon_0\min((k+\ga t)^{-\frac{1}{2}},(k+\ga t)^{-\frac{1-\a}{2}})| u(t)|_{1,\a;f(t),(k+\ga t)^{\frac{1}{2}}}.
\end{align*}
This shows that
\begin{align*}
\sup_{0< t\leq T}|u(t)|_{1,\a;f(t),(k+\ga t)^{1/2}} \leq C\Big(&|u(0)|_{1,\a;f(0)}+\Lambda_k(T,F_1,F_2,{G},f,h)\\
&+\varepsilon_0\sup_{0\leq t\leq T}|u(t)|_{1,\a;f(t),(k+\ga t)^{1/2}}\Big),
\end{align*}
which gives the desired result by taking $\e_0$ so that $C\varepsilon_0\leq \frac{1}{2}.$
\end{proof}

\subsection{Weighted $C^{1,\al}$ estimates for the transport-diffusion equation}
We consider the transport-diffusion equation with general form
\ben\label{eq:TD-E}
\partial_t u+Z\cdot\nabla u-\gamma\triangle u+F_1+F_2+\partial_iG^i=0,\quad u(0,X)=u_0(X).
\een

Given the divergence free vector field $Z(t,X)\in C^1([0,T]\times \R^d)$ and $s\in [0,T]$, we define
\beno
\f d {dt}\Phi(s,t,X)=Z(t,\Phi(s,t,X)),\quad \Phi(s,s,X)=X.
\eeno
We denote by $D\Phi$ and $\na \Phi$ the matrix with the convention
\beno
(D\Phi)_{ij}=\partial_j\Phi^{i},\quad (\nabla\Phi)_{ij}=\partial_i\Phi^{j}.
\eeno
That is, $(D\Phi)=(\nabla\Phi)^{T}$. We introduce
\beno
b=(D\Phi)^{-1},\quad a=(D\Phi)^{-1}(\nabla\Phi)^{-1},\quad a_{ij}=b_{ik}b_{kj}.
\eeno
For $v(t,X)$ defined in $[0,T]\times \R^d$, we denote
\beno
v^*(t,X)\tre v(t,\Phi(s,t,X)).
\eeno
Using the formulas
\beno
(\text{div}G)\circ\Phi=\text{div}((D\Phi)^{-1} G\circ\Phi),\ \ (\triangle u)\circ\Phi=\text{div}((D\Phi)^{-1}(\nabla\Phi)^{-1} \nabla u\circ\Phi),
\eeno
we can tranform (\ref{eq:TD-E})  into the following form
\ben\label{eq:TD-L}
\partial_t u^*(t)-\gamma\partial_i\big(a_{ij}\partial_j u^*(t)\big)+F_1^*+F_2^*+\partial_iG_*^i=0,
\een
where $G_*^i=b_{ij}(G^*)^j$.
\smallskip

We introduce the weight function $f(t,X), \widehat{f}(t,X), h(t,X)$, which satisfy (\ref{ass:W}) with a uniform constant $c_0$ and
\ben\label{ass:W2}
\begin{split}
&\int_0^tH(2\ga(t-s))h_{\pm}(s,X)ds\leq c_0^{-1}\widehat{f}(t,X)\quad \text{for all } 0\leq t\leq T,\, X\in \R^d,\\
&\int_0^T f_{\pm}(t,X\pm B_0t)dt=\int_0^T f(t,X\pm 2B_0t)dt\leq c_0^{-1},\\
&H(2\ga(t-s))\widehat{f}(s,X)\le c_0^{-1}\widehat{f}(t,X)\quad \text{for all } 0\le s\le t\le T,\, X\in \R^d,
\end{split}
\een
where we denote
\beno
f_{\pm}(t, X)=U(\pm t)f(t,X), \quad U(t)f(s,X)=f(s,X+B_0t).
\eeno

\begin{proposition}\label{prop:tran-HW1}
There exists $\varepsilon_1>0$ and $C>0$ independent of $\ga$ and $T$ such that if
$$|Z(t)+ B_0|_{1,\a;f_-(t),(1+\ga t)^{1/2}}<\varepsilon_1,$$ and (\ref{ass:W2}) holds for the minus sign,
then it holds that for $k=0,1,$
\beno
\sup_{0\leq t\leq T}|u(t)|_{1,\a;\widehat{f}_+(t),(k+\ga t)^{1/2}}\leq C\big(|u_0|_{1,\a;\widehat{f}(0),k}+\Lambda_k(T,F_1,F_2,{G},\widehat{f}_+,h)\big).
\eeno
Similarly, if
$$|Z(t)- B_0|_{1,\a;f_+(t),(1+\ga t)^{1/2}}<\varepsilon_1,$$ and (\ref{ass:W2}) holds for the plus sign,
then it holds that for $k=0,1,$
\beno
\sup_{0\leq t\leq T}|u(t)|_{1,\a;\widehat{f}_-(t),(k+\ga t)^{1/2}}\leq C\big(|u_0|_{1,\a;\widehat{f}(0),k}+\Lambda_k(T,F_1,F_2,{G},\widehat{f}_-,h)\big).
\eeno
\end{proposition}

\begin{proof}
We only consider the case of $|Z(t)+B_0|_{1,\a;f_-(t),(1+\ga t)^{1/2}}<\varepsilon_1$.  In this case, similar to \eqref{claim:phi}, we have
\beno
|\Phi(s,t,X)+B_0(t-s)-X|<2\quad \text{for}\quad  0\leq t\leq s\leq T.
\eeno
Then we get by (\ref{ass:W}) and (\ref{ass:W2})  that
\begin{align}
& \sup_{t\leq s\leq T}|\nabla Z(s)|_{0,\a;f_-(s)}\int_{0}^Tf_-\big(s, \Phi(T,s,X)\big) ds\nonumber\\
&\leq \varepsilon_1(1+\ga t)^{-1/2}c_0^{-1}\int_{0}^Tf_-\big(s, X-B_0(s-T)\big) ds\leq \varepsilon_1(1+\ga t)^{-1/2}c_0^{-1},\label{eq:w-p3}
 \end{align}
 and by \eqref{ass:W},
 \begin{align}\label{eq:w-p1}
 &U(s-t)h(t)=h(t,X+B_0(s-t))\ge c_0h(t)\circ\Phi(s,t),\\
 &U(s)\widehat{f}(t)=U(s-t)\widehat{f}_{+}(t)\ge c_0\widehat{f}_{+}(t)\circ\Phi(s,t).\label{eq:w-p2}
\end{align}

Now we fix $s\ge 0$ and assume $0\leq t\leq s\leq T$. With \eqref{eq:w-p3}, we infer from Lemma \ref{lem:char}  that
\ben\label{eq:Phi-p1}
|\nabla \Phi(s,t)-Id|_{0,\a}\leq C\varepsilon_1(1+\ga t)^{-1/2}.
\een
This implies that
\beno
|a_{ij}(t)-\delta_{ij}|_{0,\a}\leq C\varepsilon_1(1+\ga t)^{-1/2},\quad |b_{ij}(t)|_{0,\a;1,(1+\ga t)^{1/2}}\leq C.
\eeno
Using \eqref{ass:W2}, it is easy to verify that
\beno
&&H(2\ga(t-\tau))U(s)\widehat{f}(\tau,X)=U(s)H(2\ga(t-\tau))\widehat{f}(\tau,X)\leq c_0^{-1}U(s)\widehat{f}(t),
\eeno
and
\begin{align*}
\int_0^tH(2\ga(t-\tau))U(s-\tau)h(\tau,X)d\tau=&\int_0^tH(2\ga(t-\tau))U(s)h_{-}(\tau,X)d\tau\\
\leq& c_0^{-1}U(s)\widehat{f}(t).
\end{align*}
Therefore, if we take $\varepsilon_1>0$ so that $C\varepsilon_1\leq \varepsilon_0$, then we can apply Proposition \ref{prop:parabolic} to obtain
\beno
&&\sup_{0< t\leq s}|u^*(t)|_{1,\a;U(s)\widehat{f}(t),(k+\ga t)^{1/2}}\\
&&\leq C\big(|u_0\circ\Phi(s,0)|_{1,\a;U(s)\widehat{f}(0),k}+\Lambda_k(s,F_1^*,F_2^*,{G_*},U(s)\widehat{f},U(s-\cdot)h)\big).
\eeno

Thanks to \eqref{eq:Phi-p1}, we get by Lemma \ref{lem:comp1}, (\ref{eq:w-p1}) and \eqref{eq:w-p2} that
\begin{align*}
&|u_0\circ\Phi(s,0)|_{1,\a;U(s)\widehat{f}(0),k}\leq C|u_0\circ\Phi(s,0)|_{1,\a;\widehat{f}(0)\circ\Phi(s,0),k}\leq C|u_0|_{1,\a;\widehat{f}(0),k},\\
&|F_2^*(t)|_{0,\a;U(s)\widehat{f}(t)}\leq C|F_2^*(t)|_{0,\a;\widehat{f}_+(t)\circ\Phi(s,t)}\leq C|F_2(t)|_{0,\a;\widehat{f}_+(t)},\\
&|F_1^*(t)|_{1,\a;U(s-t)h(t),(k+\ga t)^{1/2}}\leq C|F_1^*(t)|_{1,\a;h(t)\circ\Phi(s,t),(k+\ga t)^{1/2}}\leq C|F_1(t)|_{1,\a;h(t),(k+\ga t)^{1/2}},
\end{align*}
and
\begin{align*}
|G_{*}(t)|_{0,\a;U(s)\widehat{f}(t),(k+\ga t)^{1/2}}\leq& C|G_{*}(t)|_{0,\a;\widehat{f}_+(t)\circ\Phi(s,t),(k+\ga t)^{1/2}}\\
\leq& C|b(t)|_{0,\a;1, (1+\ga t)^{1/2}}|G(t)\circ\Phi(s,t)|_{0,\a;\widehat{f}_+(t)\circ\Phi(s,t),(k+\ga t)^{1/2}}\\
\leq& C|G(t)|_{0,\a;\widehat{f}_+(t),(k+\ga t)^{1/2}}.
\end{align*}
This proves
\beno
\Lambda_k(s,F_1^*,F_2^*,{G_*},U(s)\widehat{f},U(s-t)h)\leq C\Lambda_k(s,F_1,F_2,{G},\widehat{f}_+,h).
\eeno
Therefore, we conclude
\beno
\sup_{0< t\leq s}|u^*(t)|_{1,\a;U(s)\widehat{f}(t),(k+\ga t)^{1/2}}
\le C\big(|u_0|_{1,\a;\widehat{f}(0),k}+\Lambda_k(s,F_1,F_2,{G},\widehat{f}_+,h)\big).
\eeno
Thanks to $u^*(s)=u(s)$ and $U(s)\widehat{f}(s)=\widehat{f}_+(s)$, we have \begin{align*}
|u(s)|_{1,\a;\widehat{f}_+(s),(k+\ga s)^{1/2}}\leq C\big(|u_0|_{1,\a;\widehat{f}(0),k}+\Lambda_k(s,F_1,F_2,{G},\widehat{f}_+,h)\big)
\end{align*} for all $0< s\leq T$. The case of $s=0$ is trivial. This completes the proof.
\end{proof}

\subsection{Main result}

Let us first introduce the weight functions
\beno
f(t)=H(1+2\mu_1t)\phi_1,\quad f_1(t)=H(1+2\mu_1t)\phi_0,
\eeno
where if $B_0=(1,0,\cdots,0)$, we may take
\ben\label{eq:weight}
\phi_1(X)=|x_1^2+x_2^2|^{-\frac{1+\delta}{2}},\quad\phi_0(X)=|x_2|^{-\delta}
\een
for some $0<\delta<\underline{\frac{1}{2}}$.  Let
\beno
g(t,X)\tre\int_{\R^d}\frac{f_+(t,Y)f_-(t,Y)}{1+|X-Y|^{d+1}}dY.
\eeno

We introduce
\beno M_{\pm}(t)\tre\sup_{0\leq \tau\leq t}\Big(|z_{\pm}^{(1)}(\tau)|_{1,\a;f_{\pm}(\tau),(1+\mu_1 \tau)^{1/2}}+|J_{\pm}^{(1)}(\tau)|_{1,\a;f_{\pm}(\tau),(\mu_1 \tau)^{1/2}}+\mu_1^{-1}|\psi_{\pm}^{(2)}(\tau)|_{1,\a;f_{1}(\tau),(\mu_1 \tau)^{1/2}}\Big).
\eeno

Main result of this section is stated as follows.

\begin{theorem}\label{thm:viscous}
Let $\al\in (0,1)$.
There exists $\epsilon_2>0$ so that if $M_\pm(0)+\mu_2/\mu_1\le \epsilon\le \epsilon_2$,
then there exists a global in time unique solution $(z_+,z_-)\in L^\infty((0,+\infty)\times\R^d)$ with the pressure $p$ determined by {(\ref{pressure-Rd})} to  the viscous MHD equations (\ref{eq:MHD-v1}) satisfying
\beno
M_\pm(t)\le C\epsilon\quad \text{for any}\quad t\in [0,+\infty).
\eeno
\end{theorem}
\begin{remark}
Thanks to $M_\pm(0)\sim \big|z_\pm(0)\langle (x_1,x_2)\rangle^{1+\delta}\big|_{1,\al}$, this means that the initial data decays at infinity only in two directions. This is a key point for the
global well-posedness in the strip domain,
especially in $\R^2$ and $\R^2\times [0,1]$.
\end{remark}
To proceed, we need to verify that the weight functions introduced here
satisfy some key properties \eqref{ass:W1} and \eqref{ass:W2}.

With the choice of \eqref{eq:weight},  it is easy to check that for $k=0,1$,
\beno
&C^{-1}R^d\min(\phi_k(X),R^{-k-\delta})\leq\int_{B(X,R)}\phi_k(Y)dY\leq CR^d\min(\phi_k(X),R^{-k-\delta}),\\
&\int_{\R}\phi_1(X+B_0t)dt\leq C\phi_0(X),
\eeno
which imply
\ben\label{f}
&C^{-1}\min(\phi_1(X),(1+\mu_1t)^{-\frac{1+\delta}{2}})\leq f(t,X)\leq C\min(\phi_1(X),(1+\mu_1t)^{-\frac{1+\delta}{2}}),\\
&\label{f1} C^{-1}\min(\phi_0(X),(1+\mu_1t)^{-\frac{\delta}{2}})\leq f_1(t,X)\leq C\min(\phi_0(X),(1+\mu_1t)^{-\frac{\delta}{2}}),\\
&\int_{\R}f(t,X+B_0s)ds\leq Cf_1(t,X).
\een
Therefore,
\ben\label{f01}\int_{B(X,R)}f_1(t,Y)dY\leq CR^d\min(R^{-\delta},(1+\mu_1t)^{-\frac{1+\delta}{2}}),\een
and
\ben\label{f02}\int_{B(X,R)}h(Y)dY\leq Ch(X),
\een
which is true for $h=1,f(t),f_1(t)$, and $f_{\pm}(t)$ by translation. Thus,
\ben\label{prop:f+-}
\int_{\R^d}\frac{f_{\pm}(t,Y)dY}{R^{d+1}+|X-Y|^{d+1}}\leq CR^{-1}f_{\pm}(t,X).\een

\begin{lemma}\label{lem:w1}
(1) The weight functions $f(t,X),\ {f_1}(t,X),\ g(t,X)$  satisfies \eqref{ass:W1} with $R=(1+\mu_1t)^\f12$  and a uniform constant $c_0$ independent of $t$.

(2) Property (\ref{ass:W2}) with $\ga=\mu_1$ holds true for $(\widehat{f},h)=(f,g) $ or $(\widehat{f},h)=(f_1,f_-) $ for the minus sign  or $(\widehat{f},h)=(f_1,f_+) $ for the plus sign.
\end{lemma}

\begin{proof}
We deduce from \eqref{f} and \eqref{f1}  that $f(t), {f_1}(t)$ satisfy \eqref{ass:W1} with $R=(1+\mu_1t)^\f12$. So do $f_{\pm}(t)$ and $f_{+}(t)f_{-}(t)$, thus $g(t)$. This also implies
\beno
g(t,X)\geq C^{-1}f_{+}(t,X)f_{-}(t,X).
\eeno

By definition, we have
\beno
H(2\mu_1(t-s)){f}(s,X)={f}(t,X),\quad H(2\mu_1(t-s)){f_1}(s,X)={f_1}(t,X)
\eeno
which give the third inequality of (\ref{ass:W2}).

Thanks to
$$\int_0^T f_{\pm}(t,X\pm B_0t)dt=\int_0^T f(t,X\pm 2B_0t)dt\leq Cf_1(t,X)\leq C,
$$
which gives the second inequality of (\ref{ass:W2}).

Thanks to
\beno\int_0^tH(2\mu_1(t-s))f_{--}(s,X)ds=\int_0^tH(2\mu_1(t-s))U(-2s)f(s,X)ds\\=\int_0^tf(t,X-2B_0s)ds\leq C{f}_1(t,X),
\eeno
which gives the first inequality of (\ref{ass:W2}) with minus sign for $(\widehat{f},h)=(f_1,f_-) $. Similarly, the first inequality of (\ref{ass:W2}) with plus sign for $(\widehat{f},h)=(f_1,f_+) $ is true.

Notice that
\begin{align*}
H(2\mu_1(t-s))g_{\pm}(s,X)=&\int_{\R^d}\frac{H(2\mu_1(t-s))(f_+(s)f_-(s))(Y\pm B_0s)}{1+|X-Y|^{d+1}}dY\\
=&\int_{\R^d}\frac{H(2\mu_1(t-s))(f(s)U(\pm 2s)f(s))(Y)}{1+|X-Y|^{d+1}}dY.
\end{align*}
By \eqref{f}, we have
\beno
f(t,X)\leq C\big(1+{|Y-X|}/{\sqrt{1+\mu_1 t}}\big)^{1+\delta}f(t,Y),
\eeno
which gives
$$f(s)U(\pm 2s)f(s)(X)\leq C\big(1+{|Y-X|}/{\sqrt{1+\mu_1 s}}\big)^{2+2\delta}f(s)U(\pm 2s)f(s)(Y).
$$
Therefore, for $t/2\leq s< t,$
\beno
&&H(2\mu_1(t-s))(f(s)U(\pm 2s)f(s))(Y)\\
&&=\int_{\R^d}K(2\mu_1(t-s),X-Y)f(s)U(\pm 2s)f(s))(X)dX\\
&&\leq C\int_{\R^d}K(2\mu_1(t-s),X-Y)\big(1+{|Y-X|}/{\sqrt{1+\mu_1 s}}\big)^{2+2\delta}f(s)U(\pm 2s)f(s)(Y)dX\\
&&\leq C f(s)U(\pm 2s)f(s))(Y)\leq C f(t)U(\pm 2s)f(0))(Y),
\eeno
and for $0\leq s\leq t/2,$
\begin{align*}
H(2\mu_1(t-s))(f(s)U(\pm 2s)f(s))\leq& C H(2\mu_1 t)(f(s)U(\pm 2s)f(s))\\
\leq& C H(2\mu_1 t)(f(0)U(\pm 2s)f(0)),
\end{align*}
therefore,
\beno
&&\int_0^tH(2\mu_1(t-s))(f(s)U(\pm 2s)f(s))ds\\
&&\leq C\int_0^{\frac{t}{2}}H(2\mu_1 t)(f(0)U(\pm 2s)f(0))ds+C\int_{\frac{t}{2}}^tf(t)U(\pm 2s)f(0))ds\\
&& \leq C H(2\mu_1 t)(f(0)f_1(0))+Cf(t)f_1(0)\leq Cf(t).
\eeno
This shows that
\beno
\int_0^tH(2\mu_1(t-s))g_{\pm}(s,X)ds\leq C\int_{\R^d}\frac{f(t,Y)}{1+|X-Y|^{d+1}}dY\leq Cf(t,X),\eeno
which gives the first inequality of (\ref{ass:W2}) for $(\widehat{f},h)=(f,g) $.
\end{proof}

\subsection{Proof of Theorem \ref{thm:viscous}}

The following lemma gives the relation between the H\"{o}lder norms of
$z_\pm^{(i)}(t), i=1,2$ and $M_\pm(t)$.

\begin{lemma}\label{lem:z1-z2}
It holds that
\beno
&&|z_{\pm}^{(2)}(t)|_{0,\a;f_{1}(t),(\mu_1 t)^{1/2}} \leq C\mu_1\min((\mu_1 t)^{-\frac{1}{2}},(\mu_1 t)^{-\frac{1-\a}{2}})M_{\pm}(t),\\
&&|z_{\pm}^{(2)}(t)|_{0,\a; 1, (1+\mu_1 t)^{1/2}}\leq C\mu_1\min((\mu_1 t)^{-\frac{1}{2}},(\mu_1 t)^{-\frac{1+\delta}{2}})M_{\pm}(t),\\
&&|\nabla z_{\pm}^{(1)}|_{1,\a;f_{\pm}(t),(\mu_1 t)^{1/2}}\leq CM_{\pm}(t).\eeno
\end{lemma}
\begin{proof}
As $z_{\pm}^{(2)}=\text{div}\psi_{\pm}^{(2)}$,  we have
\begin{align*}
|z_{\pm}^{(2)}(t)|_{0,\a;f_{1}(t),(\mu_1 t)^{1/2}}&\leq C|\nabla\psi_{\pm}^{(2)}(t)|_{0,\a;f_{1}(t),(\mu_1 t)^{1/2}}\\
&\leq C|\psi_{\pm}^{(2)}(t)|_{1,\a;f_{1}(t),(\mu_1 t)^{1/2}}\min((\mu_1 t)^{-\frac{1}{2}},(\mu_1 t)^{-\frac{1-\a}{2}})\\
&\leq C\mu_1\min((\mu_1 t)^{-\frac{1}{2}},(\mu_1 t)^{-\frac{1-\a}{2}})M_{\pm}(t),
\end{align*}
which along with \eqref{f1} gives
\begin{align*}
 |z_{\pm}^{(2)}(t)|_{0,\a;1, (1+\mu_1 t)^{1/2}}&\leq |z_{\pm}^{(2)}(t)|_{0,\a;f_{1}(t),(\mu_1 t)^{1/2}}\left(1+\frac{1}{\mu_1 t}\right)^{\frac{\a}{2}}|f_{1}(t)|_0\\
 &\leq C\mu_1\min((\mu_1 t)^{-\frac{1}{2}},(\mu_1 t)^{-\frac{1-\a}{2}})M_{\pm}(t)\left(1+\frac{1}{\mu_1 t}\right)^{\frac{\a}{2}}(1+\mu_1 t)^{-\frac{\delta}{2}}\\
 &\leq C\mu_1\min((\mu_1 t)^{-\frac{1}{2}},(\mu_1 t)^{-\frac{1+\delta}{2}})M_{\pm}(t).
 \end{align*}

Obviously, we have
 \beno
 |\nabla z_{\pm}^{(1)}|_{0,\a;f_{\pm}(t)}\leq| z_{\pm}^{(1)}|_{1,\a;f_{\pm}(t),(1+\mu_1 t)^{1/2}}\leq M_{\pm}(s).
 \eeno
Thanks to $\triangle z_{\pm}^{(1)}=\text{div} J_{\pm}^{(1)}$, we have
\begin{align*}
|\triangle z_{\pm}^{(1)}|_{0,\a;f_{\pm}(t),(\mu_1 t)^{1/2}}&\leq C|J_{\pm}^{(1)}|_{1,\a;f_{\pm}(t),(\mu_1 t)^{1/2}}\min((\mu_1 t)^{-\frac{1}{2}},(\mu_1 t)^{-\frac{1-\a}{2}})\\
&\leq CM_{\pm}{(s)}\min((\mu_1 t)^{-\frac{1}{2}},(\mu_1 t)^{-\frac{1-\a}{2}}).\end{align*}
Notice that by Lemma \ref{lem:w1},
\beno
f_\pm(t,X)\le Cf_\pm(t,Y)\quad \text{if}\quad |X-Y|\le (1+\mu_1t)^\f12.
\eeno
Then we infer from Lemma \ref{Schauder}  that
\begin{align*}
 |\nabla^2 z_{\pm}^{(1)}|_{0,\a;f_{\pm}(t),(\mu_1 t)^{1/2}}\leq& C\Big(|\nabla z_{\pm}^{(1)}|_{0,\a;f_{\pm}(t)}\min((\mu_1 t)^{-\frac{1}{2}},(\mu_1 t)^{-\frac{1-\a}{2}})+|\triangle z_{\pm}^{(1)}|_{0,\a;f_{\pm}(t),(\mu_1 t)^{1/2}}\Big)\\
 \leq& CM_{\pm}(s)\min((\mu_1 t)^{-\frac{1}{2}},(\mu_1 t)^{-\frac{1-\a}{2}}).
 \end{align*}
This proves the third inequality.
\end{proof}

Now let's begin with the proof of Theorem \ref{thm:viscous}.

\begin{proof}
For fixed $\nu>0$ and $\mu>0$, the local well-posedness of the MHD equations in the weighted H\"{o}lder space can be proved by using the semigroup method and the estimates of heat operator in the weighted H\"{o}lder space(see section 6.3). Here we omit the details.
The local well-posedness of the linear equations (4.2)-(4.7) in the weighted H\"{o}lder space is also true.

The proof of global well-posedness is based on the continuity argument. Let us first assume
\ben\label{ass:Mpm}
M_{\pm}(s)<\varepsilon_1.
\een
for $\epsilon_1>0$ given by Proposition \ref{prop:tran-HW1}.  This in particular gives
\beno
|Z^{(1)}_\pm(t)\pm B_0|_{1,\a;f_\mp(t),(1+\mu_1 t)^{1/2}}<\varepsilon_1.
\eeno
Our next goal is to show that
\begin{align}
M_{+}(s)\leq C\big(M_+(0)+(M_{+}(s)+\mu_2/\mu_1)M_{-}(s)\big),\label{eq:M+}\\
M_{-}(s)\leq C\big(M_-(0)+(M_{-}(s)+\mu_2/\mu_1)M_{+}(s)\big).\label{eq:M-}
\end{align}
With the above estimates, we can deduce our result if $\epsilon_2$ is taken small enough so that
\beno
CM_\pm(0)\le C\epsilon_2<\varepsilon_1/2,\quad C^2\epsilon_2<1/2.
\eeno
This condition on $\epsilon_2$ implies that if $M_{\pm}(s)< \varepsilon_1$
 then $M_{\pm}(s)\le 2CM_\pm(0)< \varepsilon_1.$
 \smallskip

The proof of (\ref{eq:M+}) and (\ref{eq:M-}) is split into three steps.\smallskip

\no{\bf Step 1.} $C^{1,\al}$ estimate for $z_{\pm}^{(1)}$

For the system (\ref{eq:MHD-z1-z}), we apply Proposition \ref{prop:tran-HW1} to obtain
\beno
&&\sup_{0\leq t\leq s}|z_{+}^{(1)}(t)|_{1,\a;{f}_+(t),(1+\mu_1t)^{1/2}}\\
&&\leq C\Big(|z_{+}(0)|_{1,\a;\widehat{f}(0)}+\Lambda_1\big(s,I(z_-^{(1)},z_+^{(1)}),z_-^{(2)}\cdot\nabla z_+^{(1)}+I(z_-^{(2)},z_+^{(1)}),0,{f}_+,g\big)\Big).
\eeno
By (\ref{eq:I(u,w)-est1}), we have
\begin{align*}
|I(z_-^{(1)}(t),z_+^{(1)}(t))|_{1,\a;g(t),(1+\mu_1 t)^{1/2}}\leq& C|z_-^{(1)}(t)|_{1,\a;f_-(t),(1+\mu_1 t)^{1/2}}|z_+^{(1)}(t)|_{1,\a;f_+(t),(1+\mu_1 t)^{1/2}}\\
\leq& CM_{+}(s)M_{-}(s),
\end{align*}
and  by  (\ref{eq:I(u,w)-est2}) and Lemma \ref{lem:z1-z2},
\begin{align*}
|z_-^{(2)}\cdot\nabla z_+^{(1)}(t)+I(z_-^{(2)}(t),z_+^{(1)}(t))|_{0,\a;f_+(t)}\leq& C|z_-^{(2)}(t)|_{0,\a;1,(1+\mu_1t)^{1/2}}|z_+^{(1)}(t)|_{1,\a;f_+(t),(1+\mu_1 t)^{1/2}}(1+\mu_1 t)^{-\frac{1}{2}}\\
\leq& C\mu_1 M_-(s)\min((\mu_1 t)^{-\frac{1}{2}},(\mu_1 t)^{-\frac{1+\delta}{2}})M_{+}(s)(1+\mu_1 t)^{-\frac{1}{2}}\\
\leq& C\mu_1 M_{+}(s)M_-(s)\min((\mu_1 t)^{-\frac{1}{2}},(\mu_1 t)^{-1-\frac{\delta}{2}}),
\end{align*}
and  obviously,
\beno
|z_{+}(0)|_{1,\a;{f}(0)}\leq M_+(0).
\eeno
Therefore, we obtain
\beno
\sup_{0\leq t\leq s}|z_{+}^{(1)}(t)|_{1,\a;{f}_+(t),(1+\mu_1 t)^{1/2}}\leq C\big(M_+(0)+M_{+}(s)M_-(s)\big).
\eeno
Similarly, we have
\beno
\sup_{0\leq t\leq s}|z_{-}^{(1)}(t)|_{1,\a;{f}_-(t),(1+\mu_1 t)^{1/2}}\leq C\big(M_-(0)+M_{+}(s)M_-(s)\big).
\eeno

\no{\bf Step 2.} $C^{1,\al}$ estimate for $J_{\pm}^{(1)}$

For the system (\ref{eq:MHD-z1-j}), we apply Proposition \ref{prop:tran-HW1}  to obtain
\beno
\sup_{0\leq t\leq s}|J_{+}^{(1)}(t)|_{1,\a;{f}_+(t),(\mu_1 t)^{1/2}}\leq C\Big(|J_{+}(0)|_{0,\a;\widehat{f}(0)}+\Lambda_0\big(s,\nabla z_-^{(1)}\wedge\nabla z_+^{(1)},0,z_-^{(2)}\cdot\nabla z_+^{(1)},{f}_+,g\big)\Big).
\eeno
Thanks to the choice of weight functions, we have
\beno
f_-(t,X)f_+(t,X)\le Cg(t,X).
\eeno
Then by Lemma \ref{lem:z1-z2} and analogous of Lemma \ref{lem:product}, we have
 \begin{align*}
|\nabla z_-^{(1)}\wedge\nabla z_+^{(1)}(t)|_{1,\a;g(t),(\mu_1 t)^{1/2}}\leq& C|\nabla z_-^{(1)}|_{1,\a;f_-(t),(\mu_1 t)^{1/2}}|\nabla z_+^{(1)}(t)|_{1,\a;f_+(t),(\mu_1 t)^{1/2}}\\
\leq& CM_{+}(s)M_{-}(s),\\
|z_-^{(2)}\cdot\nabla z_+^{(1)}(t)|_{0,\a;f_+(t),(\mu_1 t)^{1/2}}\leq& C|z_-^{(2)}(t)|_{0,\a;1,(\mu_1 t)^{1/2}}|\nabla z_+^{(1)}(t)|_{0,\a;f_+(t),(\mu_1 t)^{1/2}}\\
\leq& C\mu_1 \min((\mu_1 t)^{-\frac{1-\a}{2}},(\mu_1 t)^{-\frac{1}{2}})M_-(s)M_{+}(s),
\end{align*}
and $|j_{+}(0)|_{0,\a;{f}(0)}\leq M_+(0).$
Therefore, we obtain
\beno
\sup_{0\leq t\leq s}|J_{+}^{(1)}(t)|_{1,\a;{f}_+(t),(\mu_1 t)^{1/2}}\leq C\big(M_+(0)+M_{+}(s)M_-(s)\big).
\eeno
Similarly, we have
\beno
\sup_{0\leq t\leq s}|J_{-}^{(1)}(t)|_{1,\a;{f}_-(t),(\mu_1 t)^{1/2}}\leq C\big(M_-(0)+M_{+}(s)M_-(s)\big).
\eeno

\no{\bf Step 3.} $C^{1,\al}$ estimate for $\psi_{\pm}^{(2)}$

For the system (\ref{eq:MHD-psi}), we apply Proposition \ref{prop:tran-HW1} to obtain
\beno
\sup_{0\leq t\leq s}|\psi_{+}^{(2)}(t)|_{1,\a;{f}_1(t),(\mu_1t)^{1/2}}\leq C\Lambda_0\big(s,\textrm{II}_2(z_-^{(1)},z_+^{(2)})-\mu_2j_-^{(1)},\textrm{II}_1(z_-^{(2)}, z_+^{(2)}),\mu_2z_-^{(2)},{f}_1,f_-\big),
\eeno
here we used the fact that $\psi_{\pm}^{(2)}(0)=0, f_{1\pm}=f_1$, and the decomposition of $ J_{\pm}$ in (\ref{eq:j-dec}). We get by Proposition\ref{prop:Riesz} and Lemma \ref{lem:z1-z2}  that
\begin{align*}
&|\textrm{II}_2(z_-^{(1)}(t),z_+^{(2)}(t))-\mu_2j_-^{(1)}(t)|_{1,\a;f_-(t),(\mu_1 t)^{1/2}}\\
&\leq C|z_-^{(1)}(t)|_{1,\a;f_-(t),(\mu_1 t)^{1/2}}|\nabla \psi_{+}^{(2)}(t)|_{1,\a;f_1(t),(\mu_1 t)^{1/2}}+\mu_2|j_-^{(1)}(t)|_{1,\a;f(t),(\mu_1 t)^{1/2}}\\&\leq C\mu_1 M_{+}(s)M_{-}(s)+\mu_2M_{-}(s),
\end{align*}
and
\begin{align*}
|\textrm{II}_1(z_-^{(2)}(t), z_+^{(2)}(t))|_{0,\a;f_1(t)}\leq& C|z_-^{(2)}(t)|_{0,\a;f_1(t),(\mu_1 t)^{1/2}}|z_-^{(2)}(t)|_{0,\a;f_1(t),(\mu_1 t)^{1/2}}(1+\mu_1t)^{-\frac{\delta}{2}}(1+(\mu_1 t)^{-\frac{\a}{2}})\\
\leq& C\mu_1^2\min((\mu_1 t)^{-1+\frac{\a}{2}},(\mu_1 t)^{-1-\frac{\delta}{2}})M_-(s)M_{+}(s),
\end{align*}
and
\beno
|\mu_2z_-^{(2)}(t)|_{0,\a;{f}_1(t),(\ga t)^{1/2}}\leq C\mu_1\mu_2 \min((\mu_1 t)^{-\frac{1}{2}},(\mu_1 t)^{-\frac{1-\a}{2}})M_-(s).
\eeno
This shows that
\beno
\sup_{0\leq t\leq s}|\psi_{+}^{(2)}(t)|_{1,\a;{f}_1(t),(\mu_1 t)^{1/2}}\leq C\big(\mu_1 M_{+}(s)+\mu_2\big)M_-(s).
\eeno
Similarly, we have
\beno
\sup_{0\leq t\leq s}|\psi_{-}^{(2)}(t)|_{1,\a;{f}_1(t),(\mu_1 t)^{1/2}}\leq C\big(\mu_1 M_{-}(s)+\mu_2\big)M_+(s).
\eeno

Summing up the estimates in Step 1-Step 3, we conclude (\ref{eq:M+})
and (\ref{eq:M-}).
\end{proof}

\section{Appendix}

\subsection{Weighted $C^{1,\al}$ estimate for the integral operator}

Recall that
\begin{align*}
T_1u\tre &\int_{\R^d}\nabla N(X-Y)\theta(|X-Y|)u(Y)dY\nonumber\\
T_{ij}w\tre &\int_{\R^d}\partial_i\partial_j\Big(\nabla N(X-Y)(1-\theta(|X-Y|))\Big)w(Y)dY,
\end{align*}where the cut-off function $\theta$ is given by (\ref{def:theta}).

\begin{lemma}\label{lem:integral}
Let $u, w\in C^{0,\al}_h(\R^d)$ with the weight $h$ satisfying \eqref{ass:W}. Then there exists a constant $C>0$ depending only on $c_0$ so that
\beno
&&|T_1u|_{1,\al;h}\le C|u|_{0,\al;h},\\
&&|T_{ij}w|_{1,\al;g}\le C|w|_{0;h},
\eeno
where $g(X)=\int_{\R^d}\frac {h(Y)} {1+|X-Y|^{d+1}}dy$. In particular, we have
\beno
|T_1u+T_{ij}w|_{1,\al;g}\le C\big(|u|_{0,\al;h}+|w|_{0,h}\big).
\eeno
\end{lemma}

\begin{proof}
Thanks to
\beno
\left|\nabla^k\partial_i\partial_j\Big(\nabla N(X-Y)\cdot(1-\theta(|x-y|))\Big)\right|\leq \dfrac{C}{1+|x-y|^{d+1}},\ k=0,1,2,
\eeno
and $h(X)\le Cg(X)$, we get
\beno
|\nabla^kT_{ij}w(X)|\leq Cg(X)|w/h|_0,
\eeno
which in particular implies
\ben\label{eq:Bij-H}
|T_{ij}w|_{1,\al;g}\le C|w/h|_0.
\een

To deal with $T_1u$, we decompose it as follows
\beno
T_1u=\sum_{k=0}^{+\infty}B_k(u),
\eeno
where
\beno
B_k(u)=\int_{\R^d}\varphi_k(X-Y)u(Y)dY,\quad
\varphi_k(X)=\nabla N(X)\cdot\big(\theta(2^k|X|)-\theta(2^{k+1}|X|)\big).
\eeno
To proceed, we need to use the following simple facts:
\beno
&&\int_{\R^d}|\varphi_k(X)|dX\leq C2^{-k},\\
&&\int_{\R^d}|\nabla\varphi_k(X)||X|^{\a}dX\leq C2^{-k\a},\\
&&\int_{\R^d}|\nabla^2\varphi_k(X)||X|^{\a}dX\leq C2^{k(1-\a)},\\
&&\varphi_k(X)=0\quad\text{for}\quad |X|>2,\, k\ge 0.
\eeno
Then we have
\ben\label{eq:Bk1}
|B_k(u)(X)|\leq \int_{\R^d}|\varphi_k(X-Y)||h(Y)|dY|u/h|_0\leq C2^{-k}h(X)|u/h|_0.
\een
Notice that
\beno
\nabla B_k(u)(X)=\int_{\R^d}\nabla\varphi_k(X-Y)(u(Y)-u(X))dY,
\eeno
from which, we deduce
\begin{align}
|\nabla B_k(u)(X)|\leq& \int_{\R^d}|\nabla\varphi_k(X-Y)||X-Y|^{\a}(h(X)+h(Y))dY|u|_{0,\a;h}\nonumber\\
\leq& C2^{-k\a}h(X)|u|_{0,\a;h}.\label{eq:Bk2}
\end{align}
Similarly, we have
\begin{align}
|\nabla^2B_k(u)(X)|\leq C2^{k(1-\a)}h(X)|u|_{0,\a;h}.\label{eq:Bk3}
\end{align}
It follows from \eqref{eq:Bk1} and \eqref{eq:Bk2} that
\beno
&&\sum_{k=0}^{+\infty}|B_k(u)(X)|\leq \sum_{k=0}^{+\infty}C2^{-k}h(X)|u/h|_0\leq Ch(X)|u/h|_0,\\
&&\sum_{k=0}^{+\infty}|\nabla B_k(u)(X)|\leq  \sum_{k=0}^{+\infty}C2^{-k\a}h(X)|u|_{0,\al}\leq Ch(X)|u|_{0,\a;h}.
\eeno
It follows from \eqref{eq:Bk2} and \eqref{eq:Bk3} that
\beno
|\nabla B_k(u)(X)-\nabla B_k(u)(Y)|\leq C2^{-k\a}(h(X)+h(Y))|u|_{0,\a;h}\min\big\{1,2^k|X-Y|\big\},
\eeno
which gives
\begin{align*}
\Big|\sum\limits_{k=0}^{+\infty}\na\big(B_k(u)(X)-B_k(u)(Y)\big)\Big|&\leq C(h(X)+h(Y))|u|_{0,\a;h}\sum\limits_{k=0}^{+\infty}2^{-k\a}\min\big\{1,2^k|X-Y|\big\}\\
&\leq C\big(h(X)+h(Y)\big)|u|_{0,\a;h}|X-Y|^{\a}.
\end{align*}
Now we can conclude that
\beno
|T_1u|_{1,\al;h}\le \Big|\sum_{k=0}^{+\infty}B_k(u)\Big|_{1,\al;h}\le C|u|_{0,\al;h}.
\eeno
This finishes the proof of  the lemma.
\end{proof}

\begin{lemma}\label{lem:B-div}
It holds that
\begin{align*}
\text{div}\,(T_1u+T_{ij}w^{ij})+u=&\int_{\R^d}\nabla N(X-Y)\cdot\nabla\theta(|X-Y|)
u(Y)dY\\&-\int_{\R^d}\partial_i\partial_j\Big(\nabla N(X-Y)\cdot\nabla\theta(|X-Y|)\Big)w^{ij}(Y)dY.
\end{align*}
\end{lemma}

\begin{proof}
With the notations in Lemma \ref{lem:integral},
a direct calculation gives
\begin{align*}
&\text{div}\,T_{ij}(w^{ij})= -\int_{\R^d}\partial_i\partial_j\Big(\nabla N(X-Y)\cdot\nabla\theta(|X-Y|)\Big)w^{ij}(Y)dY,\\
&\text{div}B_{k}(u)= \int_{\R^d}\text{div}\varphi_k(X-Y)u(Y)dY,
\end{align*}
where
\begin{align*}
&\text{div}\varphi_k(X)=\nabla N(X)\cdot\nabla\big(\theta(2^k|X|)-\theta(2^{k+1}|X|)\big)=\varphi^*_k(X)-\varphi^*_{k+1}(X),\\
&\varphi^*_k(X)=\nabla N(X)\cdot\nabla\theta(2^{k}|X|)=-c_d\frac{2^{k}\theta'(2^{k}|X|)}{|X|^{d-1}}\geq 0.
\end{align*}
Therefore,
\beno
&&\text{div}\sum_{k=0}^NB_{k}(u)+u\\
&&=\int_{\R^d}(\varphi^*_0(X-Y)-\varphi^*_{N+1}(X-Y))u(Y)dY+u(X)\\
&&=\int_{\R^d}\varphi^*_0(|X-Y|)u(Y)dY
-\int_{\R^d}\varphi^*_{N+1}(X-Y)(u(Y)-u(X))dY\\
&&\tre I^*_0-I^*_{N+1}.
\eeno
Here we used $\int_{\R^d}\varphi^*_k(X)dX=1.$ Now,
\beno
\left|I^*_{N+1}\right|\leq {[u]_{\a}}\int_{\R^d}\varphi^*_{N+1}(X-Y)|X-Y|^{\a}dY=C[u]_{\a}2^{-N\a}\longrightarrow 0,
\eeno
as $N\rightarrow +\infty$. This proves the lemma.
\end{proof}

We also introduce
\begin{align*}
T_1(u,R)\tre&\int_{\R^d}\nabla N(X-Y)\theta(|X-Y|/R)u(Y)dY,\nonumber\\
T_{ij}(w,R)\tre&\int_{\R^d}\partial_i\partial_j\Big(\nabla N(X-Y)(1-\theta(|X-Y|/R))\Big)w(Y)dY,
\end{align*}
where $N(X)$ is the Newton potential.
Let $R\ge 1$.
If  $h(X)\leq C_0h(Y)$ for $|X-Y|\leq 2R$, then we can deduce by following the proof of Lemma \ref{lem:integral}  that
\beno
&&|T_1(u,R)|_{1,\al;g,R}+|T_{ij}(w,R)|_{1,\al;g,R}\le C\big(R^2|u|_{0,\al;h,R}+|w|_{0;h}\big),
\eeno
where  $g(X)=\int_{\R^d}\frac {h(Y)} {R^{d+1}+|X-Y|^{d+1}}dy$.
Due to (\ref{prop:f+-}),  we also have
\beno
&& R^{-1}|T_1(u,R)|_{1,\al;f_{\pm}(t),R}+|T_{ij}(w,R)|_{0,\al;f_{\pm}(t),R}\le C\big(|u|_{0,\al;f_{\pm}(t),R}+R^{-1}|w|_{0;f_{\pm}(t)}\big)
\eeno
for $R=\sqrt{1+\mu_1t}.$

In particular, we have
\begin{align}
&|I(u,w)|_{1,\a;g(t),(1+\mu_1 t)^{1/2}}\nonumber\\
&\leq C\big((1+\mu_1 t)|\nabla u|_{0,\a;f_+(t),(1+\mu_1 t)^{1/2}}|\nabla w|_{0,\a;f_-(t),(1+\mu_1 t)^{1/2}}+| u|_{0;f_+(t)}| w|_{0;f_-(t)}\big)\nonumber\\
&\leq C| u|_{1,\a;f_+(t),(1+\mu_1 t)^{1/2}}| w|_{1,\a;f_-(t),(1+\mu_1 t)^{1/2}},\label{eq:I(u,w)-est1}
\end{align}
where $g,f_{\pm}$ are defined as in section 4.4.

For $\text{div} u=\text{div} w=0,$ we have
\beno
I(u,w)\tre T_1(\pa_i{u}^j\pa_j{w}^i,R)+T_{ij}( {u}^i{w}^j,R)=\pa_iT_1({u}^j\pa_j{w}^i,R)+T_{ij}({u}^i{w}^j,R).
\eeno
Therefore, we deduce
\begin{align}
&|I(u,w)|_{0,\a;f_\pm(t)}\leq C|u|_{0,\a;1,(1+\ga t)^{1/2}}|w|_{1,\a;f_\pm(t),(1+\ga t)^{1/2}}(1+\ga t)^{-\f12}.\label{eq:I(u,w)-est2}
\end{align}

\subsection{Weighted H\"{o}lder estimates for the heat  operator}
Let $H(t)$ be the heat operator given by
\beno
H(t)f(X)\eqdef \frac{1}{(4\pi t)^{d/2}}\int_{\R^d}e^{-\frac{|X-Y|^2}{4t}}f(Y)dY=\int_{\R^d}K(t,X-Y)f(Y)dY,
\eeno
where $K(t,X)=(4\pi t)^{-d/2}e^{-\frac{|X|^2}{4t}}$.
Let $\al\ge 0$ and $k\in \N$. It is easy to verify the following properties
\ben\label{eq:ker}
\begin{split}
&|\nabla^kK(t,X)|\leq Ct^{-\frac{k}{2}}K(2t,X),\\
&|\nabla^kK(t,X)||X'|^{\a}\leq Ct^{-\frac{k-\a}{2}}K(2t,X),\\
&|\nabla^kK(t,X)-\nabla^kK(t,Y)|\leq Ct^{-\frac{k+1}{2}}K(2t,X)|X-Y|,\\
&|\nabla^kK(t,X)-\nabla^kK(t,Y)||X'|^{\a}\leq Ct^{-\frac{k+1-\a}{2}}K(2t,X)|X-Y|,
\end{split}
\een
for any $ X',Y\in B(X, \sqrt{t})$. Here $C$ is a constant independent of $t$.

We introduce the following seminorm
$$
[u]_{1;h}\eqdef \sup_{X,Y\in \R^d}\frac{| u(X)- u(Y)|}{(h(X)+h(Y))|X-Y|}.
$$
Then it is easy to check that \ben\label{22} [u]_{\a;h}\leq [u]_{1;h}^{\a}|u|_{0;h}^{1-\a},\quad |\nabla u|_{0;h}\leq 2[u]_{1;h}.\een
\begin{lemma}\label{lem:heat1}
Let $u\in C^{0,\al}_h(\R^d)$ with $0<h<C_0$ and $\al\in (0,1)$. Then there exists a constant  $C>0$ depending only on $d,\a, k$ so that for $k\in \N$,
\beno
&& |\nabla^kH(t)u|_{0;H(2t)h}\le Ct^{-\frac{k}{2}}|u|_{0;h},\\
&&[\nabla^kH(t)u]_{1;H(2t)h}\le Ct^{-\frac{k+1}{2}}|u|_{0;h},\\ &&[\nabla^kH(t)u]_{\a;H(2t)h}\le Ct^{-\frac{k}{2}}[u]_{\a;h},\\
&& [\nabla^kH(t)u]_{1;H(2t)h}\le Ct^{-\frac{k+1-\a}{2}}[u]_{\a;h}.
\eeno
\end{lemma}
\begin{proof}
Thanks to  (\ref{eq:ker}),  we have
 \begin{align*} \
 |\nabla^kH(t)u(X)|&=\left|\int_{\R^d}\nabla^kK(t,X-Y)u(Y)dY\right|\\
 &\leq \int_{\R^d}|\nabla^kK(t,X-Y)||u(Y)|dY\\
 &\leq Ct^{-\frac{k}{2}}\int_{\R^d}K(2t,X-Y)h(Y)dY |u|_{0;h}\\
 &\leq Ct^{-\frac{k}{2}}H(2t)h(X) |u|_{0;h},
 \end{align*}
which gives the first inequality.

If $|X-Y|< \sqrt{t},$ then we get by (\ref{eq:ker}) that
\begin{align*}
|\nabla^kH(t)u(X)-\nabla^kH(t)u(Y)|&=\left|\int_{\R^d}(\nabla^kK(t,X-X')-\nabla^kK(t,Y-X'))u(X')dX'\right|\\
&\leq \int_{\R^d}|\nabla^kK(t,X-X')-\nabla^kK(t,Y-X')||u(X')|dX'\\
&\leq Ct^{-\frac{k+1}{2}}|X-Y|\int_{\R^d}K(2t,X-X')h(X')dX' |u|_{0;h} \\
&\leq Ct^{-\frac{k+1}{2}}|X-Y|H(2t)h(X) |u|_{0;h},
\end{align*}
and if $|X-Y|\geq \sqrt{t},$ then
\begin{align*}
|\nabla^kH(t)u(X)-\nabla^kH(t)u(Y)|&\leq|\nabla^kH(t)u(X)|+|\nabla^kH(t)u(Y)|\\
&\leq Ct^{-\frac{k}{2}}H(2t)h(X) |u|_{0;h}+Ct^{-\frac{k}{2}}H(2t)h(Y) |u|_{0;h}\\
&\leq Ct^{-\frac{k+1}{2}}|X-Y|(H(2t)h(X)+H(2t)h(Y)) |u|_{0;h},
\end{align*}
which imply the second inequality.

For any $X,Y\in \R^d$, we have
\begin{align*}
\big|\nabla^kH(t)u(X)-\nabla^kH(t)u(Y)\big|=&\left|\int_{\R^d}\nabla^kK(t,X')u(X-X')dX'-\int_{\R^d}\nabla^kK(t,X')u(Y-X')dX'\right|\\ \leq& \int_{\R^d}|\nabla^kK(t,X')||u(X-X')-u(Y-X')|dX'\\ \leq& Ct^{-\frac{k}{2}}\int_{\R^d}K(2t,X')(h(X-X')+h(Y-X'))dX' |X-Y|^{\a}[u]_{\a;h}\\ \leq& Ct^{-\frac{k}{2}}(H(2t)h(X)+H(2t)h(Y)) |X-Y|^{\a}[u]_{\a;h},
\end{align*}
which gives the third inequality.

For any $X,Y\in\R^d$, if $|X-Y|< \sqrt{t},$ we take $Y'\in B(X,\sqrt{t})$ so that
 \beno
 h(Y')\int_{B(X,\sqrt{t})}K(2t,X-X')dX'\leq \int_{B(X,\sqrt{t})}K(2t,X-X')h(X')dX'\leq H(2t)h(X),
 \eeno
which gives $h(Y')\leq CH(2t)h(X)$. Then we deduce for $|X-Y|< \sqrt{t}$,\begin{align*}
&|\nabla^kH(t)u(X)-\nabla^kH(t)u(Y)|\\
&=\left|\int_{\R^d}(\nabla^kK(t,X-X')-\nabla^kK(t,Y-X'))(u(X')-u(Y'))dX'\right|\\
&\leq \int_{\R^d}|\nabla^kK(t,X-X')-\nabla^kK(t,Y-X')||u(X')-u(Y')|dX'\\
&\leq \int_{\R^d}|\nabla^kK(t,X-X')-\nabla^kK(t,Y-X')||X'-Y'|^{\a}(h(X')+h(Y'))dX'[u]_{\a;h}\\
&\leq Ct^{-\frac{k+1-\a}{2}}|X-Y|\int_{\R^d}K(2t,X-X')(h(X')+h(Y'))dX' [u]_{\a;h}\\
&\leq Ct^{-\frac{k+1-\a}{2}}|X-Y|(H(2t)h(X)+h(Y')) [u]_{\a;h}\\
&\leq Ct^{-\frac{k+1-\a}{2}}|X-Y|H(2t)h(X) [u]_{\a;h}.
\end{align*}
While, if $|X-Y|\geq \sqrt{t},$ then
\begin{align*}
|\nabla^kH(t)u(X)-\nabla^kH(t)u(Y)| \leq& Ct^{-\frac{k}{2}}\big(H(2t)h(X)+H(2t)h(Y)\big) |X-Y|^{\a}[u]_{\a;h}\\
\leq& Ct^{-\frac{k+1-\a}{2}}(H(2t)h(X)+H(2t)h(Y)) |X-Y|[u]_{\a;h}.
\end{align*}
This proves the fourth inequality.
\end{proof}

\begin{lemma}\label{lem:heat2}
Let $\ga>0, k\geq0,$ and $u\in C^{0,\al}_h(\R^d)$ with $0<h<C_0$. Let $R\geq\sqrt{t}>0$. Then there exists a constant $C>0$ depending only on $d,\al$ so that
\begin{align*}
& |H( t)u|_{1,\a;H(2 t)h,\sqrt{k+t}}\leq C|u|_{1,\a;h,\sqrt{k}},\\
& |H( t)u|_{1,\a;H(2 t)h,R}\leq C\varphi_{\a}( R)/\varphi_{\a}(  \sqrt{t})|u|_{0,\a;h},
\end{align*}
where $\varphi_{\a}( R)=\max(R,R^{1+\a})$
\end{lemma}
\begin{proof}
By Lemma \ref{lem:heat1} and (\ref{22}), we have
\beno
&&|H( t)u|_{0;H(2 t)h}\leq C|u|_{0;h},\ [H( t)u]_{\a;H(2 t)h}\leq C[u]_{\a;h},\ |H( t)u|_{0,\a;H(2 t)h}\leq C|u|_{0,\a;h},\\
&&|\nabla H( t)u|_{0;H(2 t)h}\leq \min(Ct^{-\frac{1}{2}}|u|_{0;h},Ct^{-\frac{1-\a}{2}}[u]_{\a;h})\leq C\min(t^{-\frac{1}{2}},t^{-\frac{1-\a}{2}})|u|_{0,\a;h},\\
&&[\nabla H( t)u]_{\a;H(2 t)h}\leq \min(Ct^{-\frac{1+\a}{2}}|u|_{0;h},Ct^{-\frac{1}{2}}[u]_{\a;h})\leq C\min(t^{-\frac{1+\a}{2}},t^{-\frac{1}{2}})|u|_{0,\a;h}.
\eeno
Due to $\nabla H( t)u=H( t) \nabla u$, we have
$$|\nabla H( t)u|_{0;H(2 t)h}\leq C|\nabla u|_{0;h},\quad [\nabla H( t)u]_{\a;H(2 t)h}\leq C[\nabla u]_{\a;h}.$$
Therefore,
\begin{align*}
 |H( t)u|_{1,\a;H(2 t)h,\sqrt{k+t}}=&|H( t)u|_{0,\a;H(2 t)h}+\max((k+t)^{\frac{1-\a}{2}},(k+t)^{\frac{1}{2}})|\nabla H( t)u|_{0;H(2 t)h}\\
 &+\max((k+t)^{\frac{1}{2}},(k+t)^{\frac{1+\a}{2}})[\nabla H( t)u]_{\a;H(2 t)h}\\
\leq& C|u|_{0,\a;h}+\max(k^{\frac{1-\a}{2}},k^{\frac{1}{2}})|\nabla H( t)u|_{0;H(2 t)h}+\max(t^{\frac{1-\a}{2}},t^{\frac{1}{2}})|\nabla H( t)u|_{0;H(2 t)h}\\ &+\max(k^{\frac{1}{2}},k^{\frac{1+\a}{2}})[\nabla H( t)u]_{\a;H(2 t)h}+\max(t^{\frac{1}{2}},t^{\frac{1+\a}{2}})[\nabla H( t)u]_{\a;H(2 t)h}\\
\leq& C|u|_{0,\a;h}+C\max(k^{\frac{1-\a}{2}},k^{\frac{1}{2}})|\nabla u|_{0;h}+C|u|_{0,\a;h}\\
&+C\max(k^{\frac{1}{2}},k^{\frac{1+\a}{2}})[\nabla u]_{\a;h}+C|u|_{0,\a;h}\leq C|u|_{1,\a;h,\sqrt{k}},
\end{align*}
which gives the first inequality. Also,
\begin{align*}
|H( t)u|_{1,\a;H(2 t)h,R}=&|H( t)u|_{0,\a;H(2 t)h}+\max(R^{1-\a},R)(|\nabla H( t)u|_{0;H(2 t)h}+R^{\a}[\nabla H( t)u]_{\a;H(2 t)h})\\
 \leq& C|u|_{0,\a;h}+\max(R,R^{1+\a})(t^{-\frac{\a}{2}}|\nabla H( t)u|_{0;H(2 t)h}+[\nabla H( t)u]_{\a;H(2 t)h})\\
 \leq& C|u|_{0,\a;h}+C\varphi_{\a}( R)\min(t^{-\frac{1+\a}{2}},t^{-\frac{1}{2}})|u|_{0,\a;h}\\
 \leq& C\varphi_{\a}( R)/\varphi_{\a}(  \sqrt{t})|u|_{0,\a;h},
\end{align*}
which gives the second inequality.
\end{proof}

\subsection{Riesz transform in the weighted H\"{o}lder spaces}

Throughout this subsection, we take $f,f_1,f_\pm$ be as in section 4.4.
We need the following property for the weight functions.
\begin{lemma}\label{lem:weight2} For $h=1,f_1(t),f(t),f_{\pm}(t),$ we have
\ben\label{f2}
R^{-d}\int_{B(X,R)}h(Y)f_1(t,Y)dY\leq Ch(X)\min(R^{-\delta},(1+\mu_1 t)^{-\frac{\delta}{2}}) .
\een
\end{lemma}
\begin{proof}
 The case  of $h=1$ follows from \eqref{f01}. We denote
 \beno
 \rho_1(X)=|x_2|,\quad \rho_2(X)=|(x_1,x_2)|\quad \text{for}\,X=(x_1,\cdots,x_d)\in\R^d.
  \eeno
Then by \eqref{f1}, {for $h=f_1(t)$} if $\rho_1(X)\geq 2R $ or $\rho_1(X)\leq 2\sqrt{1+\mu_1 t}, $  we have
\beno
h(Y)\leq Ch(X)\quad \text{for}\,\, |Y-X|\leq R,
\eeno
which gives,
\begin{align*}
R^{-d}\int_{B(X,R)}h(Y)f_1(t,Y)dY\leq& CR^{-d}\int_{B(X,R)}h(X)f_1(t,Y)dY\\\leq& C h(X)\min(R^{-\delta},(1+\mu_1 t)^{-\frac{\delta}{2}}) .
\end{align*}
Using \eqref{f}, the above inequality holds for $h=f(t)$ if $\rho_2(X)\geq 2R $ or $\rho_2(X)\leq 2\sqrt{1+\mu_1 t}$.

For the case of $h=f_1(t)$, if $2\sqrt{1+\mu_1 t}\leq\rho_1(X)\leq 2R $, then by \eqref{f1},
\beno
&&h(X)\geq C^{-1}\phi_1(X)\geq C^{-1}R^{-\delta},\\
 &&h(Y)f_1(t,Y)\leq C\phi_1(Y)^2=C\rho_1(Y)^{-2\delta},
\eeno
which imply
\beno
R^{-d}\int_{B(X,R)}h(Y)f_1(t,Y)dY\leq CR^{-d}\int_{B(X,R)}\rho_1(Y)^{-2\delta}dY\leq CR^{-2\delta}\leq Ch(X)R^{-\delta}.
\eeno

For the case of $h=f(t)$,  if $2\sqrt{1+\mu_1 t}\leq\rho_2(X)\leq 2R $,
then by \eqref{f},
\beno
&&h(X)\geq C^{-1}\phi_2(X)\geq C^{-1}R^{-1-\delta},\\
&&h(Y)f_1(t,Y)\leq C\phi_1(Y)\phi_2(Y)=C|y_1|^{-\frac{1}{2}-\delta}|y_2|^{-\frac{1}{2}-\delta},
\eeno
which imply
\beno
R^{-d}\int_{B(X,R)}h(Y)f_1(t,Y)dY\leq CR^{-1-2\delta}\leq Ch(X)R^{-\delta}.
\eeno
Therefore, \eqref{f2} is true for $h=f_1(t), f(t)$. The case of $h=f_{\pm}(t)$ follows from the case of $h=f(t)$ by translation.
\end{proof}

\begin{proposition}\label{prop:Riesz}
It holds that
\begin{align*}
&|[u,R_iR_j]\partial_kw|_{1,\a;f_\pm(t),(\mu_1 t)^{1/2}}\leq C|u|_{1,\a;f_\pm(t),(1+\mu_1 t)^{1/2}}|w|_{1,\a;f_1(t),(\mu_1 t)^{1/2}},\\
&|R_iR_j(uw)|_{0,\a;f_1(t)}\leq C(1+\mu_1 t)^{-\frac{\delta}{2}}(1+(\mu_1 t)^{-\frac{\a}{2}})|u|_{0,\a;f_1(t),(\mu_1 t)^{1/2}}|w|_{0,\a;f_1(t),(\mu_1 t)^{1/2}}.
\end{align*}
\end{proposition}

The proof of the proposition is very complicated. Let us begin with some
reductions. For fixed $i,j$,  we have
\beno
R_iR_jw(X)+\frac{\delta_{ij}}{d}w(X)=-p.v.\int_{\R^d}\partial_i\partial_jN(X-Y)w(Y)dY\tre\sum_{n=-\infty}^{\infty}R_{ij}^n(w),
\eeno
 where
\beno
R_{ij}^n(u)=-\int_{\R^d}\varphi_n(X-Y)u(Y)dY
\eeno
with $\varphi_n(X)=\partial_i\partial_jN(X)\big(\theta(2^n|X|)-\theta(2^{n+1}|X|)\big)$. Therefore,
\ben\label{eq:com}
[u,R_iR_j]\partial_kw=\sum_{n=-\infty}^{\infty}[u,R_{ij}^n]\partial_kw.
\een
\begin{lemma}\label{Riesz1} For $h=1,f_1(t),f(t),f_{\pm}(t),$  it holds that
\beno
|R_iR_j(u)|_{0,\a;h,(1+\mu_1 t)^{1/2}}\leq C(1+\mu_1 t)^{-\frac{\delta}{2}}|u|_{0,\a;hf_1(t),(1+\mu_1 t)^{1/2}}.
\eeno
\end{lemma}
\begin{proof}
Notice that
 $$\int_{R^d}\varphi_n(X)dX=0,\ \text{supp}\varphi_n\subset B(0,2^{1-n})\setminus B(0,2^{-1-n}),\ \ \ |\nabla^{l}\varphi_n|\leq C2^{n(d+l)},\ l=0,1,2,$$
we deduce from Lemma \ref{lem:weight2} that
\begin{align*}
|R_{ij}^n(u)(X)|\leq& \int_{\R^d}|\varphi_n(X-Y)|h(Y)f_1(t,Y)dY |u|_{0;hf_1(t)}\\ \leq& C2^{nd}\int_{B(X,2^{1-n})}h(Y)f_1(t,Y)dY |u|_{0;hf_1(t)}\\
\leq& C2^{n\delta}h(X) |u|_{0;hf_1(t)},
\end{align*}

For $X\in \R^d,$ we have
\beno
R_{ij}^n(u)(X)=-\int_{\R^d}\varphi_n(X-Y)(u(Y)-u(X))dY,
\eeno
which along with \eqref{f02} gives
\begin{align*}
 |R_{ij}^n(u)(X)|&\leq \int_{\R^d}|\varphi_n(X-Y)|(h(X)+h(Y))|X-Y|^{\a}dY [u]_{\a;h}\\
 &\leq C2^{n(d-\a)}\int_{B(X,2^{1-n})}(h(X)+h(Y))dY [u]_{\a;h}\\
 &\leq C2^{-n\a}h(X) [u]_{\a;h}.
\end{align*}
By \eqref{f1}, we have
\beno
[u]_{\a;h}\leq C(1+\mu_1 t)^{-\frac{\delta}{2}}[u]_{\a;hf_1(t)}\leq C(1+\mu_1 t)^{-\frac{\a+\delta}{2}}|u|_{0,\a;hf_1(t),(1+\mu_1 t)^{1/2}}.
\eeno
Thus, we can conclude
\begin{align*}
|R_{i}R_{j}(u)(X)|\leq&\sum_{n=-\infty}^{\infty}|R_{ij}^n(u)(X)|\\
\leq& \sum_{n=-\infty}^{\infty}C\min(2^{n\delta},2^{-n\a}(1+\mu_1 t)^{-\frac{\a+\delta}{2}})h(X) |u|_{0,\a;hf_1(t),(1+\mu_1 t)^{1/2}}\\
\leq& C(1+\mu_1 t)^{-\frac{\delta}{2}}h(X) |u|_{0,\a;hf_1(t),(1+\mu_1 t)^{1/2}}.
\end{align*}

For any $X,X'\in \R^d,\ |X-X'|\leq 2^{-n},$
\begin{align*}
|R_{ij}^n(u)(X)-R_{ij}^n(u)(X')|\leq& \int_{\R^d}|\varphi_n(X-Y)-\varphi_n(X'-Y)|(h(X)+h(Y))|X-Y|^{\a}dY [u]_{\a;h}\\
\leq& C2^{n(d+1-\a)}|X-X'|\int_{B(X,2^{1-n})}(h(X)+h(Y))dY [u]_{\a;h}\\
\leq& C2^{n(1-\a)}|X-X'|h(X) [u]_{\a;h},
\end{align*}
which gives for any $X,X'\in \R^d,$
\beno |R_{ij}^n(u)(X)-R_{ij}^n(u)(X')|\leq C2^{-n\a}\min(1,2^n|X-X'|)(h(X)+h(X')) [u]_{\a;h}.
\eeno
Then we have
\begin{align*}
|R_{i}R_{j}(u)(X)-R_{i}R_{j}(u)(X')|\leq&\sum_{n=-\infty}^{\infty}|R_{ij}^n(u)(X)-R_{ij}^n(u)(X')|\\
\leq& \sum_{n=-\infty}^{\infty}C2^{-n\a}\min(1,2^n|X-X'|)(h(X)+h(X')) [u]_{\a;h}\\ \leq& C|X-X'|^{\a}(h(X)+h(X')) [u]_{\a;h},
\end{align*}
which implies $[R_{i}R_{j}u]_{\a;h}\leq C [u]_{\a;h}$. Thus,
\begin{align*}
|R_iR_j(u)|_{0,\a;h,(1+\mu_1 t)^{1/2}}=&|R_iR_j(u)|_{0;h}+(1+\mu_1 t)^{\a/2}[R_{i}R_{j}u]_{\a;h}\\ \leq&  C(1+\mu_1 t)^{-\frac{\delta}{2}} |u|_{0,\a;hf_1(t),(1+\mu_1 t)^{1/2}}+(1+\mu_1 t)^{\a/2}[u]_{\a;h}\\ \leq&  C(1+\mu_1 t)^{-\frac{\delta}{2}} |u|_{0,\a;hf_1(t),(1+\mu_1 t)^{1/2}} ,
\end{align*}
which gives our result.
\end{proof}

\begin{lemma}\label{Riesz2}
For $l=0,1,$ it holds that
\beno
|\nabla^l[u,R_{ij}^n]\partial_kw(X)|\leq C2^{n(l-\a)}|\nabla u|_{0;B(X,2^{1-n})} [w]_{\a}.
\eeno
\end{lemma}
\begin{proof}
Thanks to
\ben\label{f6}
&&[u,R_{ij}^n]\partial_kw(X)=\int_{\R^d}\varphi_n(X-Y)(u(Y)-u(X))\partial_kw(Y) dY \nonumber \\
&&=\int_{\R^d}\partial_k\varphi_n(X-Y)(u(Y)-u(X))w(Y)dY-\int_{\R^d}\varphi_n(X-Y)\partial_ku(Y)w(Y) dY,\een
we deduce that
\begin{align*}
|[u,R_{ij}^n]\partial_kw(X)|\leq& \int_{\R^d}|\partial_k\varphi_n(X-Y)||X-Y|dY |\nabla u|_{0;B(X,2^{1-n})}|w|_{0;B(X,2^{1-n})}\\&+\int_{\R^d}|\varphi_n(X-Y)| dY|\nabla u|_{0;B(X,2^{1-n})}|w|_{0;B(X,2^{1-n})}\\
\leq& C|\nabla u|_{0;B(X,2^{1-n})}|w|_{0;B(X,2^{1-n})},
\end{align*}

Thanks to
\ben\label{f7}
&&\nabla[u,R_{ij}^n]\partial_kw(X) =\int_{\R^d}\nabla\partial_k\varphi_n(X-Y)(u(Y)-u(X))w(Y)dY\nonumber\\&&\quad-\nabla u(X)\int_{\R^d}\partial_k\varphi_n(X-Y)w(Y)dY-\int_{\R^d}\nabla\varphi_n(X-Y)\partial_ku(Y)w(Y) dY,\een
we can similarly deduce that
\beno |\nabla[u,R_{ij}^n]\partial_kw(X)|\leq  C2^n|\nabla u|_{0;B(X,2^{1-n})}|w|_{0;B(X,2^{1-n})}.
\eeno

As $[u,R_{ij}^n]\partial_kw=[u,R_{ij}^n]\partial_k(w-w(X)), $ we have for $l=0,1,$
\begin{align*}
|\nabla^l[u,R_{ij}^n]\partial_kw(X)|\leq& C2^{nl}|\nabla u|_{0;B(X,2^{1-n})}|w-w(X)|_{0;B(X,2^{1-n})}\\\leq& C2^{n(l-\a)}|\nabla u|_{0;B(X,2^{1-n})} [w]_{\a}.
\end{align*}
This completes the proof.
\end{proof}

\begin{lemma}\label{Riesz3}
If  $|u|_{1,\a;h,(1+\mu_1 t)^{1/2}}=|w|_{1,\a;f_1(t),(\mu_1 t)^{1/2}}=1$ for
$h=1,f_1(t),f(t),f_{\pm}(t)$, then we have
\beno
&&|[u,R_{ij}^n]\partial_kw(X)|\leq C h(X)\min(2^{n\delta}(1+\mu_1 t)^{-\frac{1}{2}},2^{-n\a}(1+\mu_1 t)^{-\frac{1+\delta+\a}{2}}),\\
&&|\partial_l[u,R_{ij}^n]\partial_kw(X)|\leq C h(X)(1+\mu_1 t)^{-\frac{1+\delta+\a}{2}}\min(2^{n(1-\a)},2^{-n\a}(\mu_1 t)^{-\frac{1}{2}}).
\eeno
\end{lemma}

\begin{proof}
As $|u|_{1,\a;h,(1+\mu_1 t)^{1/2}}=|w|_{1,\a;f_1(t),(\mu_1 t)^{1/2}}=1,$ we have
$$|u(X)|\leq h(X),\ |\nabla u(X)|\leq h(X)(1+\mu_1 t)^{-\frac{1}{2}},\ |w(X)|\leq f_1(t,X).$$
Using $f_1(t,X)\leq C(1+\mu_1 t)^{-\frac{\delta}{2}},$ we also have
\beno
&&|w|_0\leq C(1+\mu_1 t)^{-\frac{\delta}{2}},\quad [w]_{\a}\leq C(1+\mu_1 t)^{-\frac{\delta}{2}},\\
&&|\nabla w|_0\leq C(1+\mu_1 t)^{-\frac{\delta}{2}}(\mu_1 t)^{-\frac{1}{2}},\\&&[\nabla w]_{\a}\leq C(1+\mu_1 t)^{-\frac{\delta}{2}}\min((\mu_1 t)^{-\frac{1}{2}},(\mu_1 t)^{-\frac{1+\a}{2}})\leq C(1+\mu_1 t)^{-\frac{\delta+\a}{2}}(\mu_1 t)^{-\frac{1}{2}},
\eeno
and
$$[w]_{\a}\leq C|w|_0^{1-\a}|\nabla w|_0^{\a}\leq C(1+\mu_1 t)^{-\frac{\delta}{2}}(\mu_1 t)^{-\frac{\a}{2}}.$$
Therefore
$$[w]_{\a}\leq C(1+\mu_1 t)^{-\frac{\delta}{2}}\min(1,(\mu_1 t)^{-\frac{\a}{2}})\leq C(1+\mu_1 t)^{-\frac{\delta+\a}{2}}.$$
Then we deduce from \eqref{f6} and Lemma \ref{lem:weight2} that for $2^{-n}\geq \sqrt{1+\mu_1 t},$
\begin{align*}
|[u,R_{ij}^n]\partial_kw(X)|\leq& \int_{\R^d}|\partial_k\varphi_n(X-Y)|(h(X)+h(Y))f_1(t,Y)dY \\&+(1+\mu_1 t)^{-\frac{1}{2}}\int_{\R^d}|\varphi_n(X-Y)|h(Y)f_1(t,Y) dY \\ \leq& C2^{n(d+1)}\int_{B(X,2^{1-n})}(h(X)+h(Y))f_1(t,Y)dY \\&+C(1+\mu_1 t)^{-\frac{1}{2}}2^{nd}\int_{B(X,2^{1-n})}h(Y)f_1(t,Y) dY \\ \leq& C2^{n(1+\delta)}h(X) +C(1+\mu_1 t)^{-\frac{1}{2}}2^{n\delta}h(X)\\
\leq& C(1+\mu_1 t)^{-\frac{1}{2}}2^{n\delta}h(X).
\end{align*}

For $2^{-n}\leq \sqrt{1+\mu_1 t},$ we have
 $$|\nabla u|_{0;B(X,2^{1-n})}\leq |\nabla u|_{0;h, B(X,2^{1-n})}|h|_{0;B(X,2^{1-n})}\leq Ch(X)(1+\mu_1 t)^{-\frac{1}{2}},$$
 where we used the fact that $h$ satisfies \eqref{ass:W1} with $R=\sqrt{1+\mu_1 t}$.
Similarly, we have
 $$
 [\nabla u]_{\a;B(X,2^{1-n})}\leq [\nabla u]_{\a;h, B(X,2^{1-n})}|h|_{0;B(X,2^{1-n})}\leq Ch(X)(1+\mu_1 t)^{-\frac{1+\a}{2}}.
 $$

Then we get by Lemma \ref{Riesz2} that
\begin{align*}
|[u,R_{ij}^n]\partial_kw(X)|\leq& C2^{n\a}h(X)|\nabla u|_{0;B(X,2^{1-n})} [w]_{\a}\\\leq& C2^{n\a}h(X) (1+\mu_1 t)^{-\frac{1+\delta+\a}{2}},
\end{align*}
which gives the first inequality of the lemma.

Similarly, by \eqref{f7}, Lemma \ref{lem:weight2} and Lemma \ref{Riesz2},
we can deduce
\begin{align*}
|\partial_l[u,R_{ij}^n]\partial_kw(X)|\leq& C h(X)2^n\min(2^{n\delta}(1+\mu_1 t)^{-\frac{1}{2}},2^{-n\a}(1+\mu_1 t)^{-\frac{1+\delta+\a}{2}})\\
\leq& C h(X)2^{n(1-\a)}(1+\mu_1 t)^{-\frac{1+\delta+\a}{2}}.
\end{align*}

On the other hand,
\begin{align*}
\partial_l[u,R_{ij}^n]\partial_kw(X)=&\int_{\R^d}\partial_l\varphi_n(X-Y)(u(Y)-u(X))\partial_kw(Y) dY\\&-\partial_lu(X)\int_{\R^d}\varphi_n(X-Y)\partial_kw(Y) dY\\
\tre& [u,\partial_lR_{ij}^n]\partial_kw(X)+\partial_lu(X)R_{ij}^n\partial_kw(X).
\end{align*}

From the proof of Lemma \ref{Riesz1}, we can see that
$$|R_{ij}^n\partial_kw(X)|\leq C2^{-n\a}[\partial_kw]_{\a}\leq C2^{-n\a}(1+\mu_1 t)^{-\frac{\delta+\a}{2}}(\mu_1 t)^{-\frac{1}{2}}.$$
By \eqref{f02}, we deduce that for $2^{-n}\geq \sqrt{1+\mu_1 t},$
\begin{align*}
|[u,\partial_lR_{ij}^n]\partial_kw(X)|\leq&\int_{\R^d}|\partial_l\varphi_n(X-Y)|(|u(Y)|+|u(X)|)|\partial_kw(Y) | dY\\ \leq& C2^{n(d+1)}(1+\mu_1 t)^{-\frac{\delta}{2}}(\mu_1 t)^{-\frac{1}{2}}\int_{B(X,2^{1-n})}(h(Y)+h(X)) dY\\ \leq& C2^{n}(1+\mu_1 t)^{-\frac{\delta}{2}}(\mu_1 t)^{-\frac{1}{2}}h(X)\\
\leq& C2^{-n\a}(1+\mu_1 t)^{-\frac{1+\delta+\a}{2}}(\mu_1 t)^{-\frac{1}{2}}h(X).
\end{align*}
For $2^{-n}\leq \sqrt{1+\mu_1 t},$ using the formula
\begin{align*}
[u,\partial_lR_{ij}^n]\partial_kw(X)=&\int_{\R^d}\partial_l\varphi_n(X-Y)(u(Y)-u(X))(\partial_kw(Y)-\partial_kw(X)) dY\\&+\partial_kw(X)\int_{\R^d}\varphi_n(X-Y)(\partial_lu(Y)-\partial_lu(X)) dY,
\end{align*}
we deduce that
\begin{align*}
 |[u,\partial_lR_{ij}^n]\partial_kw(X)|\leq&\int_{\R^d}|\partial_l\varphi_n(X-Y)||X-Y|^{1+\a} dY|\nabla u|_{0;B(X,2^{1-n})}[\partial_kw]_{\a}\\&+|\partial_kw|_{0}\int_{\R^d}\varphi_n(X-Y)|X-Y|^{\a} dY [\nabla u]_{\a;B(X,2^{1-n})}\\ \leq& C2^{-n\a}h(X)(1+\mu_1 t)^{-\frac{1}{2}}(1+\mu_1 t)^{-\frac{\delta+\a}{2}}(\mu_1 t)^{-\frac{1}{2}}\\&+C(1+\mu_1 t)^{-\frac{\delta}{2}}(\mu_1 t)^{-\frac{1}{2}}2^{-n\a}h(X)(1+\mu_1 t)^{-\frac{1+\a}{2}}\\ \leq& C2^{-n\a}h(X)(1+\mu_1 t)^{-\frac{1+\delta+\a}{2}}(\mu_1 t)^{-\frac{1}{2}}.
 \end{align*}
This shows that
 \begin{align*}
 |\partial_l[u,R_{ij}^n]\partial_kw(X)|\leq& [u,\partial_lR_{ij}^n]\partial_kw(X)|+|\partial_lu(X)R_{ij}^n\partial_kw(X)|\\ \leq& C2^{-n\a}h(X)(1+\mu_1 t)^{-\frac{1+\delta+\a}{2}}(\mu_1 t)^{-\frac{1}{2}}\\&+C h(X)(1+\mu_1 t)^{-\frac{1}{2}}2^{-n\a}(1+\mu_1 t)^{-\frac{\delta+\a}{2}}(\mu_1 t)^{-\frac{1}{2}}\\ \leq& C2^{-n\a}h(X)(1+\mu_1 t)^{-\frac{1+\delta+\a}{2}}(\mu_1 t)^{-\frac{1}{2}},
 \end{align*}
 which gives  the second inequality of the lemma.
 \end{proof}

 Using the formula
 \begin{align*}
 \partial_m[u,\partial_lR_{ij}^n]\partial_kw(X)=&\int_{\R^d}\partial_m\partial_l\varphi_n(X-Y)(u(Y)-u(X))\partial_kw(Y) dY\\&-\partial_mu(X)\int_{\R^d}\partial_l\varphi_n(X-Y)\partial_kw(Y) dY,
\end{align*}
we can also deduce that
\ben\label{f9} |\partial_m[u,\partial_lR_{ij}^n]\partial_kw(X)|\leq C2^{n(1-\a)}h(X)(1+\mu_1 t)^{-\frac{1+\delta+\a}{2}}(\mu_1 t)^{-\frac{1}{2}}.
\een

Now we are in position to prove Proposition \ref{prop:Riesz}.

\begin{proof}
We get by Lemma \ref{Riesz1} with $h=f_1(t)$ that
\begin{align*}
|R_iR_j(uw)|_{0,\a;f_1(t)}\leq& C(1+\mu_1 t)^{-\frac{\delta}{2}}|uw|_{0,\a;f_1(t)^2,(1+\mu_1 t)^{1/2}}\\ \leq& C(1+\mu_1 t)^{-\frac{\delta}{2}}(1+(\mu_1 t)^{-\frac{\a}{2}})|uw|_{0,\a;f_1(t)^2,(\mu_1 t)^{1/2}}\\ \leq& C(1+\mu_1 t)^{-\frac{\delta}{2}}(1+(\mu_1 t)^{-\frac{\a}{2}})|u|_{0,\a;f_1(t),(\mu_1 t)^{1/2}}|w|_{0,\a;f_1(t),(\mu_1 t)^{1/2}},
\end{align*}
which gives the second inequality of the proposition.

For the first inequality, without lose of generality, we can assume
$$|u|_{1,\a;h,(1+\mu_1 t)^{1/2}}=|w|_{1,\a;f_1(t),(\mu_1 t)^{1/2}}=1,$$
here $h=f_\pm(t)$.

First of all, by Lemma \ref{Riesz3}, we have
\begin{align*}
\big|[u,R_iR_j]\partial_kw(X)\big|\leq& \sum_{n=-\infty}^{\infty}|[u,R_{ij}^n]\partial_kw(X)|\\ \leq& C\sum_{n=-\infty}^{\infty}h(X)\min(2^{n\delta}(1+\mu_1 t)^{-\frac{1}{2}},2^{-n\a}(1+\mu_1 t)^{-\frac{1+\delta+\a}{2}})\\
\leq& C h(X)(1+\mu_1 t)^{-\frac{1+\delta}{2}},
\end{align*}
and
\begin{align*}
\big|\partial_l[u,R_iR_j]\partial_kw(X)\big|\leq& \sum_{n=-\infty}^{\infty}|\partial_l[u,R_{ij}^n]\partial_kw(X)|\\ \leq& C\sum_{n=-\infty}^{\infty}h(X)(1+\mu_1 t)^{-\frac{1+\delta+\a}{2}}\min(2^{n(1-\a)},2^{-n\a}(\mu_1 t)^{-\frac{1}{2}})\\ \leq& C h(X)(1+\mu_1 t)^{-\frac{1+\delta+\a}{2}}(\mu_1 t)^{-\frac{1-\a}{2}}.
\end{align*}

Now we consider $X,Y\in\R^d,\ |X-Y|\leq \sqrt{1+\mu_1 t}$. It follows from Lemma \ref{Riesz3} that
$$\big|[u,R_{ij}^n]\partial_kw(X)-[u,R_{ij}^n]\partial_kw(Y)\big|\leq Ch(X)2^{-n\a}(1+\mu_1 t)^{-\frac{1+\delta+\a}{2}}\min(1,2^n|X-Y|),$$
 here we used the fact that $h$ satisfies \eqref{ass:W1} with $R=\sqrt{1+\mu_1 t}$. Therefore,
 \begin{align*}
\big|[u,R_iR_j]\partial_kw(X)-[u,R_iR_j]\partial_kw(Y)\big|\leq& \sum_{n=-\infty}^{\infty}|[u,R_{ij}^n]\partial_kw(X)-[u,R_{ij}^n]\partial_kw(Y)|\\ \leq& C\sum_{n=-\infty}^{\infty}h(X)2^{-n\a}(1+\mu_1 t)^{-\frac{1+\delta+\a}{2}}\min(1,2^n|X-Y|)\\\leq& C h(X)(1+\mu_1 t)^{-\frac{1+\delta+\a}{2}}|X-Y|^{\a}.
\end{align*}

We write
$$\partial_l[u,R_iR_j]\partial_kw=[u,\partial_lR_iR_j]\partial_kw+\partial_lu\ R_iR_j\partial_kw,$$
where
$$[u,\partial_lR_iR_j]\partial_kw=\sum_{n=-\infty}^{\infty}[u,\partial_lR_{ij}^n]\partial_kw.$$
We get by Lemma \ref{Riesz3} and \eqref{f9}  that
 $$|[u,\partial_lR_{ij}^n]\partial_kw(X)-[u,\partial_lR_{ij}^n]\partial_kw(Y)|\leq Ch(X)2^{-n\a}(1+\mu_1 t)^{-\frac{1+\delta+\a}{2}}(\mu_1 t)^{-\frac{1}{2}}\min(1,2^n|X-Y|),$$
 which gives
 \begin{align*}
 \big|[u,\partial_lR_iR_j]\partial_kw(X)-[u,\partial_lR_iR_j]\partial_kw(Y)\big|\leq& \sum_{n=-\infty}^{\infty}|[u,\partial_lR_{ij}^n]\partial_kw(X)-[u,\partial_lR_{ij}^n]\partial_kw(Y)|\\ \leq& C\sum_{n=-\infty}^{\infty}h(X)2^{-n\a}(1+\mu_1 t)^{-\frac{1+\delta+\a}{2}}(\mu_1 t)^{-\frac{1}{2}}\min(1,2^n|X-Y|)\\ \leq& C h(X)(1+\mu_1 t)^{-\frac{1+\delta+\a}{2}}(\mu_1 t)^{-\frac{1}{2}}|X-Y|^{\a},
\end{align*}
 Thanks to
\begin{align*}
|\partial_kw|_{0,\a;f_1(t),(1+\mu_1 t)^{1/2}}\leq& (1+(\mu_1 t)^{-\frac{\a}{2}})|\partial_kw|_{0,\a;f_1(t),(\mu_1 t)^{1/2}}\\ \leq& (1+(\mu_1 t)^{-\frac{\a}{2}})\min((\mu_1 t)^{-\frac{1-\a}{2}},(\mu_1 t)^{-\frac{1}{2}})\leq 2(\mu_1 t)^{-\frac{1}{2}},
\end{align*}
we infer from Lemma \ref{Riesz1} that
\begin{align*}
|\partial_lu\ R_iR_j\partial_kw|_{0,\a;h,(1+\mu_1 t)^{1/2}}\leq& C|\partial_lu|_{0,\a;h,(1+\mu_1 t)^{1/2}}|R_iR_j\partial_kw|_{0,\a;1,(1+\mu_1 t)^{1/2}}\\ \leq& C(1+\mu_1 t)^{-\frac{1}{2}}(1+\mu_1 t)^{-\frac{\delta}{2}}|\partial_kw|_{0,\a;h,(1+\mu_1 t)^{1/2}}\\
\leq& C(1+\mu_1 t)^{-\frac{1+\delta}{2}}(\mu_1 t)^{-\frac{1}{2}},
\end{align*}
and
\beno
[\partial_lu\ R_iR_j\partial_kw]_{\a;h}\leq C(1+\mu_1 t)^{-\frac{1+\delta+\a}{2}}(\mu_1 t)^{-\frac{1}{2}}.
\eeno
This shows that
\beno|\partial_l[u,R_iR_j]\partial_kw(X)-\partial_l[u,R_iR_j]\partial_kw(Y)|\leq C h(X)(1+\mu_1 t)^{-\frac{1+\delta+\a}{2}}(\mu_1 t)^{-\frac{1}{2}}|X-Y|^{\a}.
\eeno

For the case of  $X,Y\in\R^d,\ |X-Y|\geq \sqrt{1+\mu_1 t},$  we have
\begin{align*}
|[u,R_iR_j]\partial_kw(X)-[u,R_iR_j]\partial_kw(Y)|\leq& C (h(X)+h(Y))(1+\mu_1 t)^{-\frac{1+\delta}{2}}\\ \leq& C (h(X)+h(Y))(1+\mu_1 t)^{-\frac{1+\delta+\al}{2}}|X-Y|^{\a},
\end{align*}
and
\begin{align*}
|\partial_l[u,R_iR_j]\partial_kw(X)-\partial_l[u,R_iR_j]\partial_kw(Y)| \leq& C (h(X)+h(Y))(1+\mu_1 t)^{-\frac{1+\delta+\a}{2}}(\mu_1 t)^{-\frac{1-\a}{2}}\\ \leq& C (h(X)+h(Y))(1+\mu_1 t)^{-\frac{1+\delta+\a}{2}}(\mu_1 t)^{-\frac{1}{2}}|X-Y|^{\a}.
\end{align*}

In summary, we conclude
\begin{align*}
&|[u,R_iR_j]\partial_kw|_{1,\a;h,(\mu_1 t)^{1/2}}=|[u,R_iR_j]\partial_kw|_{0;h}+[[u,R_iR_j]\partial_kw]_{\a;h}\\&\qquad+\max((\mu_1 t)^{\frac{1-\a}{2}},(\mu_1 t)^{\frac{1}{2}})\Big(|\nabla[u,R_iR_j]\partial_kw|_{0;h}+(\mu_1 t)^{\frac{\a}{2}}[\nabla[u,R_iR_j]\partial_kw]_{\a;h}\Big)\\
&\leq C(1+\mu_1 t)^{-\frac{1+\delta}{2}}+C(1+\mu_1 t)^{-\frac{1+\delta+\a}{2}}+C\max((\mu_1 t)^{\frac{1-\a}{2}},(\mu_1 t)^{\frac{1}{2}})\times\\
&\qquad\quad \Big((1+\mu_1 t)^{-\frac{1+\delta+\a}{2}}(\mu_1 t)^{-\frac{1-\a}{2}}+(\mu_1 t)^{\frac{\a}{2}}(1+\mu_1 t)^{-\frac{1+\delta+\a}{2}}(\mu_1 t)^{-\frac{1}{2}}\Big)\leq C,\end{align*}
which gives the first inequality of the proposition.
\end{proof}

\subsection{Weighted Schauder estimate}

Let $h(X)$ be a {positive bounded} weight satisfying
\beno
 h(X)\leq C_0h(Y)\  \text{for} \ |X-Y|\leq 2R,\ R>0.
\eeno

\begin{lemma}\label{Schauder}
Let $u\in C^{2,\al}_h(\R^d)$. Then we have
\beno
|\nabla^2 u|_{0,\a;h,R}\leq C\Big(|\nabla u|_{0,\a;h}\min(R^{-1+\a},R^{-1})+|\triangle u|_{0,\a;h,R}\Big).
\eeno
Here $C$ is a constant depending only on $C_0$.
\end{lemma}

\begin{proof}
Fix $X\in \R^d$ and consider the function $w(Y)=u(Y)-u(X)-(Y-X)\cdot\nabla u(X)$. So,
\beno
\nabla^2 w=\nabla^2 u,\quad \triangle w=\triangle u,\quad |\triangle u|_{0,\a;B(X,2R),R}\leq 2C_0h(X)|\triangle u|_{0,\a;h,R},
\eeno
where
\beno
|u|_{0,\a;B(X,2R),R}\tre |u|_{0;B(X,2R)}+R^{\a}[u]_{\a;B(X,2R)}.
\eeno

As $\nabla w(Y)=\nabla u(Y)-\nabla u(X)$, we have for $|X-Y|\leq 2R,$
\beno
&&|\nabla w(Y)|=|\nabla u(Y)-\nabla u(X)|\leq (h(X)+h(Y))|X-Y|^{\a}|\nabla u|_{0,\a;h}\leq 4C_0h(X)R^{\a}|\nabla u|_{0,\a;h},\\
&&|\nabla w(Y)|\leq|\nabla u(Y)|+|\nabla u(X)|\leq (h(X)+h(Y))|\nabla u|_{0,\a;h}\leq 2C_0h(X)|\nabla u|_{0,\a;h}.
\eeno
This shows that
\beno
|\nabla w|_{0;B(X,2R)}\leq 4C_0h(X)\min(R^{\a},1)|\nabla u|_{0,\a;h},
\eeno
from which and $w(X)=0,$ we infer
$$| w|_{0;B(X,2R)}\leq 2R|\nabla w|_{0;B(X,2R)}\leq 8C_0h(X)\min(R^{1+\a},R)|\nabla u|_{0,\a;h}.$$
Then by the (scaled) Schauder estimate, we obtain
\begin{align*}
 |\nabla^2 w|_{0,\a;B(X,R),R}\leq& C\big(R^{-2}| w|_{0;B(X,2R)}+|\triangle w|_{0,\a;B(X,2R),R}\big)\\\leq&
  Ch(X)(\min(R^{-1+\a},R^{-1})|\nabla u|_{0,\a;h}+|\triangle u|_{0,\a;h,R})\tre Ch(X)A,
 \end{align*}
which in particular shows
\beno
 |\nabla^2 u(X)|=|\nabla^2 w(X)|\leq |\nabla^2 w|_{0,\a;B(X,R),R}\leq Ch(X)A.
\eeno

On the other hand, if $|Y-X|<R,$ then
\beno
 |\nabla^2 u(X)-\nabla^2 u(Y)|\leq |X-Y|^{\a}R^{-\a}|\nabla^2 w|_{0,\a;B(X,R),R}\leq Ch(X
 )A|X-Y|^{\a}R^{-\a},
\eeno
and if $|Y-X|\geq R,$ then
 \begin{align*}
 |\nabla^2 u(X)-\nabla^2 u(Y)|&\leq |\nabla^2 u(X)|+|\nabla^2 u(Y)|\\
 &\leq Ch(X)A+Ch(Y)A\\
 &\leq C(h(X)+h(Y))A|X-Y|^{\a}R^{-\a}.
 \end{align*}
This gives
\beno
|\na^2 u|_{0,\al;h,R}=|\na^2u|_{0;h}+R^\al[\na^2 u]_{\al;h}\le CA.
\eeno
The proof is finished.
\end{proof}

\section*{Acknowledgments}

Z. Zhang is partially supported by NSF of China under Grant
11371039 and 11421101.

 \end{document}